\documentclass{amsart}
\usepackage{amsmath}
\usepackage{bbm}
\usepackage{latexsym}
\usepackage{xcolor}

 \usepackage[usenames,dvipsnames]{pstricks}
 \usepackage{epsfig}
 \usepackage{pst-grad} 
 \usepackage{pst-plot} 

\newtheorem{thm}{Theorem}
\newtheorem{lemma}[thm]{Lemma}
\newtheorem{cor}[thm]{Corollary}
\newtheorem{prop}[thm]{Proposition}

\newtheorem*{propi}{Proposition}
\newtheorem*{thmi}{Theorem}

\newcommand*\pFqskip{8mu}
\catcode`,\active
\newcommand*\pFq{\begingroup
        \catcode`\,\active
        \def ,{\mskip\pFqskip\relax}%
        \dopFq
}
\catcode`\,12
\def\dopFq#1#2{%
        {}_{#1}F_{#2}
        \endgroup
}

\theoremstyle{remark} 
\newtheorem{remark}[]{Remark}
\newtheorem{example}[]{Example}

\newcommand{\ra}{\rightarrow}
\newcommand{\RR}{\mathbb{R}}
\newcommand{\RP}{\mathbb{R}\textrm{P}}
\newcommand{\EE}{\mathbb{E}}
\newcommand{\PP}{\mathbb{P}}
\newcommand{\CC}{\mathbb{C}}

\newcommand{\e}{\varepsilon}
\newcommand{\f}{f_{\beta,n}}
\newcommand{\Mel}{\mathcal{M}_{n-1}^+}
\newcommand{\Mellin}{\mathcal{M}_{n}^+}
\newcommand{\Span}{\text{span}}
\newcommand{\op}{\textrm{op}}

\newcommand{\N}{\mathcal{N}}
\newcommand{\ii}{\textrm{i}^+}

\title[Gap probabilities and applications]{Gap probabilities and applications to geometry and random topology}
\author{Antonio Lerario}
\author{Erik Lundberg}

\begin{document}

\maketitle
\begin{abstract}We give an exact formula for the value of the derivative \emph{at zero} of the gap probability $f_{\beta, n}$ in finite Gaussian $\beta$-ensembles ($\beta=1,2,4$). As $n$ goes to infinity our computation provides:
$$f_{\beta, n}'(0)\sim -\frac{2\sqrt{2}}{\pi}n^{1/2}.$$

As a first application of the above computation, we consider the set $\Sigma_{\beta, n}$ of $\beta$-Hermitian matrices with Frobenius norm one and determinant zero 
(i.e. the set of $n\times n$ norm-one, singular, \emph{symmetric} matrices for $\beta=1$, \emph{hermitian} for $\beta=2$ and \emph{quaternionic hermitian} for $\beta=4$). We give an exact formula for its intrinsic volume; as $n$ goes to infinity this formula becomes:
$$|\Sigma_{\beta, n}|\sim |S^{N_\beta-2}|\cdot \frac{2}{\sqrt{\pi}}n^{1/2}$$
(here $N_\beta$ is the real dimension of the vector space of $\beta$-Hermitian matrices).

As a second application we consider the problem of computing Betti numbers of an intersection of $k$ random Kostlan quadrics in $\RP^n.$ We show that for each $i\geq 0$:
$$\EE b_i(X)=1+O(n^{-M})\quad \textrm{for all $M>0$}$$
In other words, each \emph{fixed} Betti number of $X$ is asymptotically expected to be one; in fact as long as $i=i(n)$ is sufficiently bounded away from $n/2$ the above rate of convergence is uniform.
In the case $k=2$ the the sum of \emph{all} Betti numbers was recently shown \cite{Lerario2012} to equal $n+o(n)$. 
Here we sharpen this asymptotic, showing that:
$$\sum_{j=0}^n\EE b_j(X)=n + \frac{2}{\sqrt{\pi}}n^{1/2} + O(n^c) \quad \textrm{for any $c>0$}$$
(the term $\frac{2}{\sqrt{\pi}}n^{1/2}\sim -\sqrt{\frac{\pi}{2}}f'_{1,n}(0)$ comes from contributions of \emph{middle} Betti numbers).
\end{abstract}

\section{Introduction}

In this paper we study a special statistic of Random Matrix Theory, the so called \emph{gap probability}, and reveal some surprising connections it has with natural problems coming from geometry and random topology. 

Specifically we consider the classical \emph{finite} $\beta$-ensembles $G_{\beta,n}$ of random matrices ($n$ is the size of the matrix) and the function:
 $$f_{\beta, n}(\e)=\textrm{probability that a matrix from $G_{\beta, n}$ has no eigenvalues in $(-\e, \e).$}$$
Here $G_{1, n}=\textrm{GOE}(n)$ (the \emph{orthogonal} ensemble),  $G_{2, n}=\textrm{GUE}(n)$ (the \emph{unitary} ensemble) and  $G_{4, n}=\textrm{GSE}(n)$ (the \emph{symplectic} ensemble); these are ensembles of random Hermitian matrices with Gaussian entries (see \cite{Fyodorov, Mehta, Tao} for more details and properties of statistics of these ensembles). 
For example the probability distribution on $\textrm{Sym}(n, \RR)$ giving rise to the ensemble $G_{1, n}$ is given by defining for every open set $U\subset \textrm{Sym}(n, \RR)$:
$$\textrm{probability that $Q$ belongs to $U$}=\frac{1}{2^{n/2}\pi^{n(n+1)/4}}\int_{U}e^{-\frac{1}{2}\textrm{tr}(Q^2)}dQ.$$

With slight abuse of notation, we will still denote by $G_{\beta, n}$ the Euclidean space of such Hermitian matrices endowed with the \emph{Frobenius} norm; in particular $G_{1,n}$ is the set of real symmetric, $G_{2,n}$ the set of Hermitian and $G_{4,n}$ the set of quaternionic Hermitian matrices. We denote by $N_\beta=n+\frac{1}{2}n(n-1)\beta$ the dimension of $G_{\beta, n}$ as a real vector space. 
The key ingredient for our applications will be the explicit computation of the derivative at zero of the gap probability; deferring the exact result for later, we give for the moment only an asymptotic formula, in order to start discussing how it can be used for applications.

\begin{propi}[An asymptotic]\label{gapasympt}The following asymptotic holds for the derivative at zero of the gap probability:
$$f'_{\beta, n}(0)\sim -\frac{2\sqrt{2}}{\pi}n^{1/2}\quad \beta=1,2,4.$$
\end{propi}
Although we only perform the asymptotic analysis for $\beta=1,2,4$, our exact formula \eqref{gapexct} (discussed later on in the Introduction) holds for arbitrary $\beta>0.$

\begin{center}
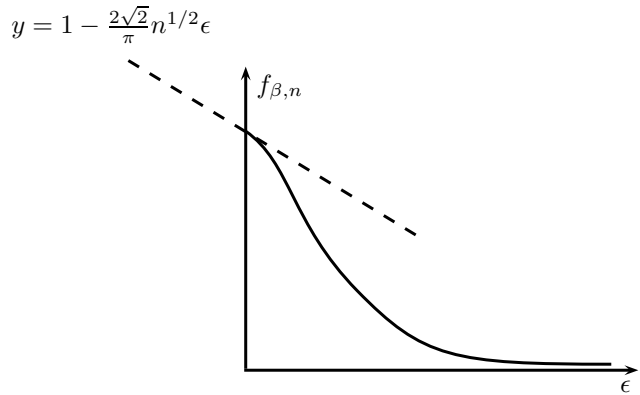
\begin{figure}\label{gapfigure}
\scalebox{1} 
{
\begin{pspicture}(0,-2.6176953)(10.82291,2.6176953)
\psbezier[linewidth=0.04](4.9810157,0.9992969)(5.5810156,0.5792969)(5.5810156,-0.21884882)(6.4610157,-1.1207031)(7.341016,-2.0225575)(7.7210155,-2.1007032)(9.841016,-2.1007032)
\psline[linewidth=0.04cm,linestyle=dashed,dash=0.16cm 0.16cm](7.2410154,-0.38070312)(3.4210157,1.9392968)
\usefont{T1}{ptm}{m}{n}
\rput(5.432471,1.5842968){$f_{\beta, n}$}
\usefont{T1}{ptm}{m}{n}
\rput(10.032471,-2.395703){$\epsilon$}
\usefont{T1}{ptm}{m}{n}
\rput(3.1924708,2.4242969){$y=1-\frac{2\sqrt{2}}{\pi}n^{1/2}\epsilon$}
\psline[linewidth=0.04cm,arrowsize=0.05291667cm 2.0,arrowlength=1.4,arrowinset=0.4]{->}(4.9810157,-2.1607032)(4.9810157,1.8592969)
\psline[linewidth=0.04cm,arrowsize=0.05291667cm 2.0,arrowlength=1.4,arrowinset=0.4]{->}(4.9610157,-2.1807032)(10.201015,-2.1807032)
\end{pspicture} 
}

\caption{A picture of the gap probability}
\end{figure}
\end{center}

\subsection*{Volume of singular Hermitian matrices}As a first application, let us consider the above space of Hermitian matrices and the set $\Sigma_{\beta, n}$ of norm-one singular matrices:
$$\Sigma_{\beta, n}=\{Q\in G_{\beta, n}\,\,\textrm{s.t.}\,\, \|Q\|^2=1 \,\,\textrm{and}\, \,\det(Q)=0\}.$$
The definition of determinant for  the case  $\beta=4$ of quaternionic matrices requires some care, see \cite{Mehta} for details.  

For example if $\beta=1$ and $n=2$ we can write $Q=\left(\begin{smallmatrix}x&y\\y&z\end{smallmatrix}\right)$ and $\Sigma_{1,2}$ is given in $G_{1,n}\simeq\RR^3$ by the two equations $xz-y^2=0$ (the determinant of $Q$ vanishes) and $x^2+2y^2+z^2=1$ (the Frobenius norm of $Q$ is one).

In general $\Sigma_{\beta, n}$ is a codimension-two algebraic subset of $G_{\beta, n}$; it is not smooth for $n\geq3$, but its singularities have codimension three in it. The topology of $\Sigma_{1,n}$ has been studied by the first author in \cite{Lerario2012}: it turns out that it is stable homotopy equivalent to a disjoint union of Grassmannians and the sum of its Betti numbers (the total number of ``holes'' it has) is $2^n$. 

Here we consider \emph{metric} properties of $\Sigma_{\beta, n}$ coming from its inclusion into the Euclidean space of Hermitian matrices with the Frobenius norm. Specifically we are interested in the following question:
\begin{equation}\label{q1}\textrm{``What is the intrinsic volume of $\Sigma_{\beta, n}$?''}\end{equation}

The corresponding problem without requiring the self-adjoint property, was studied by A. Edelman, E. Kostlan and M. Shub in \cite{EdelmanKostlan95, EdelmanKostlanShub}; in this case the set $M$ of singular $n\times n$ real matrices of norm one has volume:
\begin{equation}\label{volumegeneral}|M|=\frac{2\pi^{n^2/2}\Gamma((n+1)/2)}{\Gamma(n/2)\Gamma((n^2-1)/2)}\sim |S^{n^2-2}|\cdot \frac{\sqrt{\pi}}{2}n^{1/2}\end{equation}
(hereafter for a semialgebraic set $S$ of dimension $d$ and whose smooth locus is endowed with a Riemannian structure, we denote by $|S|$ its intrinsic $d$-dimensional volume).

Being $\Sigma_{\beta, n}$ an \emph{hypersurface} in the sphere $S^{N_\beta-1}$, the standard strategy to compute its volume is to take an $\e$-tube in $S^{N_\beta-1}$ around it and look at (one-half) the derivative at zero of the volume $v(\e)$ of this tube (see \cite{tubes} for details on this kind of constructions). Here by $\e$-tube we mean the set of points in $S^{N_\beta-1}$ at distance less then $\e$ from $\Sigma_{\beta, n}$. At this point it is important to notice that $|\det(Q)|\leq\e$ \emph{does not} define an $\e$-tube: the reason is that the norm of the gradient of the determinant along the smooth points of its zero locus is not one. 

The classical Eckart-Young theorem (see \cite[Ch. 11]{BCSS}) states that the distance, in the Frobenius norm, between an invertible matrix $X$ and the set $Z$ of degenerate matrices is given by:
$$d_F(X, Z)=\|X^{-1}\|_{\textrm{op}}^{-1}=\sigma(X)$$
where the last quantity is the \emph{least singular value} of $X$, the smallest of the modulus of the eigenvalues of $X$ . We will prove a generalization of this result for our ensembles (Theorem \ref{Eck-You} below) which allows to write the $\e$-tube of $Z$ in $G_{\beta, n}$ as $\sigma(Q)\leq \e$. In other words, the complement of the $\e$-tube of $Z$ in $G_{\beta,n}$ will consist of matrices with no eigenvalues in $(-\e, \e)$. Intersecting such tube with the sphere we get a neighborhood of $\Sigma_{\beta, n}$ whose volume, to first order, coincides with the volume of an $\e$-tube. Moreover, since the probability distribution on $G_{\beta, n}$ is uniform on the unit sphere $S^{N_\beta-1}$, the limit we are interested can be rewritten as:

 $$|\Sigma_{\beta, n}|=\lim_{\e\to 0}\frac{v(\e)}{2\e}=|S^{N_\beta-1}|\cdot \lim_{\e\to 0}\frac{1-\PP\{\sigma(Q)\geq \e\|Q\|_F\}}{2\e}.$$ 

The probability in the above statement is \emph{almost} the gap probability (geometrically is the probability of a cone, whereas the gap probability is for a cylinder); this is responsible for the Gamma functions in the following statement, which gives an answer to question \eqref{q1} above.
\begin{thmi}[The volume of singular matrices]\label{thm:golden}
For $\beta=1,2,4$:
$$|\Sigma_{\beta, n}| = |S^{N_\beta-1}|\frac{\sqrt{2}\Gamma(\frac{N_\beta}{2})}{\Gamma(\frac{N_\beta-1}{2})}\cdot \frac{1}{2}\left( - f_{\beta, n}'(0) \right).$$
\end{thmi}

Plugging in the asymptotic given in Proposition \ref{gapasympt}, we obtain the analogue of \eqref{volumegeneral}:
\begin{equation}\label{volumeasympt}|\Sigma_{\beta, n}|\sim |S^{N_\beta-2}|\cdot \frac{2}{\sqrt{\pi}}n^{1/2}.\end{equation}

A classical way to apply these results is through the \emph{integral geometry formula} (see \cite{Howard} for a modern exposition on the subject). This formula allows to reduce many problems in random geometry to simple volume computations; for example we can use it, combined with the above results, to get the expected number  $\alpha_n$ of zeros of the determinant of a random matrix polynomial $A(t)=A_0+A_1t+\cdots+A_kt^k$:
\begin{equation}\label{matrixpoly} \alpha_n \sim \frac{4}{\sqrt{\pi}} n^{1/2} \alpha_1\quad \textrm{with}\quad  \alpha_1\sim 2\log k \end{equation}
(this result should be compared with \cite[Thm. 6.1]{EdelmanKostlan95}).
\subsection*{Topology of random intersections of quadrics}As a second application, we consider the problem of studying the expectation of Betti numbers of algebraic subsets of $\RP^{n-1}$ defined by random quadratic equations.

The problem of computing the expectation of topological properties of random algebraic varieties has become very popular recently. 
The first advance dates back to the work of P. B\"urgisser \cite{Bu} and S. S. Podkorytov \cite{Po}, 
who computed the expectation of the Euler characteristic of an algebraic set defined by independent, 
centered random polynomials whose distribution is invariant by orthogonal change of variables. 

For the case of our interest, consider $q_1, \ldots, q_k$ independent random quadratic forms, distributed in such a way that the corresponding matrices $Q_1, \ldots, Q_k$ are $\textrm{GOE}(n)$ (we will call such distribution the \emph{Kostlan} distribution, see \cite{Bu, EdelmanKostlan95}). Let us denote by $X\subset \RP^{n-1}$ the common zero locus of these quadratic forms:
$$X=\{[x]\in \RP^{n-1}\,|\, q_1(x)=\cdots=q_k(x)=0\}.$$

Using the results from \cite{Bu}, one can compute (assuming $\dim(X)$ is even) the average Euler characteristic of $X$:
$$\EE \chi(X)=a_0+a_2+\cdots+a_{{\dim(X)}},$$
where the $a_{2j}$ are the coefficients of the power series $\sum_{j\geq 0}a_{2j}t^{2j}=(\frac{2}{1+t^2})^{k/2}.$ 

However, these results are still related to metric properties: 
they follow from a striking computation of the average curvature polynomial of a random algebraic variety, 
but they give little information on individual Betti numbers 
(take for example the case of a plane curve of degree $d$, 
whose Euler characteristic is always zero but whose number of components can be as large as $\frac{(d-1)(d-2)}{2}+1$).

Progress on purely topological problems (such as estimating individual Betti numbers) has been made only recently. 

The first result in this direction was made by F. L. Nazarov and M. L. Sodin: in \cite{NazarovSodin} the authors prove that a random spherical harmonic of degree $d$ on $S^2$ has on average $c\cdot d^2$ nodal lines. Extending their technique, the current authors (motivated by P. Sarnak's letter \cite{Sarnak}) studied the expectation of the number of connected components $b_0$ of a random hypersurface $Y$ of degree $d$ in $\RP^{n}$ (here a homogeneous polynomial of degree $d$ in $n+1$ is sampled at random uniformly from the unit sphere in the $L^2_{S^n}$-norm). In \cite{LerarioLundberg} they proved that there exist constants $C_n,c_n>0$ such that for large $d$:
$$ c_n d^n \leq \EE b_0(Y) \leq C_n d^{n}.$$

The novel techniques introduced in \cite{NazarovSodin} can be extended 
to show that in fact $\EE b_0(Y)/d^n$ has a limit as $d$ goes to
infinity. However the method is not explicit and 
yields little more than the non-vanishing of this limit. In the case
$n=2$, where $\EE b_0(Y)$ is  asymptotic to $c\cdot d^2$ for $c>0$, M. Nastasescu \cite{Sarnak, Nastasescu}
computed this Nazarov-Sodin constant $c$ numerically and    
found that it is approximately 0.0195.

In a sequence of papers \cite{GayetWelschinger2012, GayetWelschinger3, GayetWelschinger2013} D. Gayet and J-Y. Welschinger proved that for a Kostlan hypersurface $Y$ of degree $d$ in $\RP^n$ (i.e. whose defining polynomial has a distribution which is invariant by \emph{unitary} change of coordinates), for every Betti number $b_i$ there are two constants $a_{i,n}, A_{i, n}>0$ such that:
$$a_{i,n}d^{n/2}\leq \EE b_i(Y)\leq A_{i, n}d^{n/2}$$

Again, the problem of determining the sharp constants is far from trivial: 
the reason is that for the large degree limit no technique is available for exact computations (the upper bound is obtained using Morse inequalities and the lower bound using the barrier method from \cite{LerarioLundberg,NazarovSodin}: 
neither of these methods produces equalities).

Here we focus on the somehow opposite situation: the degree and the codimension are fixed (we consider intersection of $k$ quadrics in $\RP^{n-1}$) and the number of variables goes to infinity. The deterministic part of this problem has been studied by the first author and A. A. Agrachev in \cite{AgrachevLerario} and by the first author in \cite{Lerariotwo, Lerariocomplexity}. The main ingredient is the use of \emph{spectral sequences}, a powerful machinery from algebraic topology. The advantage of this technique is that for the case of intersection of \emph{few} quadrics (compared to the number of variables) it gives very accurate approximations to the topology of $X$.

In this setting we consider Betti numbers as the basic non metric topological invariant of $X$: they count the number of ``holes'' in $X$. The question we are interested in is:
\begin{equation}\label{q2}\textrm{``What is the expected value of the Betti numbers of $X$?''}\end{equation}
In addition to computing the expectation of a single Betti number $b_i(X)$, 
we are also interested in understanding how they distribute in the range $i=0, \ldots, \dim(X).$

A random study of a spectral sequence was done by the first author in \cite{Lerario2012}, where the topology of the intersection $X$ of \emph{two} random, independent Kostlan quadrics in $\RP^{n-1}$ was studied. The author proved that as $n$ goes to infinity:
\begin{equation}\label{eq2}\EE b(X)=n+o(n),\end{equation}
where $b(X)=\sum_{i\geq 0}b_i(X;\mathbb{Z}_2)$ denotes the sum of all Betti numbers of $X$. Here we sharpen this result as follows. 
\begin{thmi}[Intersection of two quadrics]
If $X$ is the intersection of two random, independent Kostlan quadrics in $\RP^{n-1}$:
$$\mathbb{E}b(X) = n  -\sqrt{\frac{\pi}{2}}f'_{1,n}(0) + O(n^c)= n + \frac{2}{\sqrt{\pi}}n^{1/2} + O(n^c) \quad \textrm{for any $c>0$}.$$
\end{thmi}
Thus we get two orders of precision with \emph{exact} constants. A random application of the spectral sequence argument requires an average count of the number of singular lines in the span of the two quadrics defining $X$; ultimately this leads to the appearance of the derivative of the gap probability in the above formula.

In the general case of intersection of $k$ quadrics, we can fix a Betti number $b_i$ and ask for its expectation; even allowing the index $i$ to depend on $n$, as long as it stays sufficiently bounded away from $n/2$ the expectation turns out to be $1$. Here we use the full power of spectral sequences, that guarantee in the large $n$ limit and for a fixed $i$ the existence of \emph{exactly one} homology class of dimension $i$ in $X$ (with high probability).
\begin{thmi}[Intersection of $k$ quadrics.]
For every $k, i\in \mathbb{N}$, every $M>0$ and every open set $J\subset [0,\frac{1}{2})\cup(1/2, 1]$ we have:
$$\mathbb{E}b_i(X)=1+O(n^{-M})\quad \textrm{\and}\quad \sum_{j\in nJ}\EE b_i(X)=|\mathbb{N} \cap n J|+O(n^{-M}).$$
\end{thmi}
Thus, as long as we fix the complexity class 
(intersection of $k$ quadrics) 
and the dimension of the homology, 
everything stabilizes as $n$ goes to infinity and the ``pointwise'' limit of the function $i\mapsto \EE b_i(X)$ is one. In fact we prove that in the above theorem the convergence is uniform for Betti numbers $b_i$ with $i$ bounded away from $\frac{n}{2}$ by at least $n^{\alpha}$, $0<\alpha<1$ (see Theorem \ref{thm:bettiquadrics}).

On the other hand, as $n$ goes to infinity the sum of \emph{all} Betti numbers \emph{is not} expected to be $n$. 
For example, middle Betti numbers are responsible for the $2n^{1/2}/\sqrt{\pi}$ term in the formula for two quadrics.

We conjecture that for the intersection of $k$ quadrics the middle Betti numbers give a contribution of the order $\Theta(n^{(k-1)/2}).$

\subsection*{Gap probability at zero}
We now turn back to the basic problem of evaluating exactly the derivative at zero of the gap probability. The asymptotics of a rescaled version of $f_{\beta, n}$, tracing its behavior close to zero, but letting $n$ going to infinity first is well studied (see \cite[Ch. 3]{AGZ} and the references therein). In our setting, we are actually interested in the gap probability for fixed $n$ (in particular its derivative at zero). For \emph{finite} ensembles the study goes back to the pioneering work of M. Gaudin \cite{Gaudin} and later M. Jimbo, T. Miwa, Y. M\^ori and M. Sato \cite{JMMS}. Forrester and Witte \cite{FoWy}, drawing on \cite{JMMS}, have evaluated $f_{\beta,n}$ ($\beta=1,2$) using methods from integrable systems. For example, if $\beta=1$ and $n$ is even:\footnote{Here we use the same notation as in \cite{FoWy} to help the reader comparing with this reference. The subscript of $\sigma_V$ is due to the connection with the Painlev\'e fifth equation.}
$$f_{1,n}(\e)=\tau_{\sigma_V}(\e^2),$$
where $\tau_{\sigma_V}$ is a function satisfying:
$$\label{tau}\sigma_V(t)=t\frac{d}{dt}\log \tau_{\sigma_V}(t)\quad \textrm{and}\quad \lim_{t\to 0^+}\sigma_{V}(t)t^{-1/2}=-\frac{\Gamma(\frac{n+1}{2})}{\Gamma(\frac{n}{2})\Gamma(\frac{1}{2})\Gamma(\frac{3}{2})}=-c_n. $$

Using the two above relations, one can evaluate the derivative at zero of $f_{1, n}$ under the assumption that $n$ is even (see \cite{Lerario2012}):
$$f'_{1, n}(0)=-2c_n.$$

In the general case, $\beta>0$ and regardless the parity of $n$, the situation is more delicate and we use the \emph{joint density} of the eigenvalues $\lambda_1, \ldots, \lambda_n$ for $Q\in G_{\beta, n}$. We will denote such joint density by $F_{\beta, n}$; its explicit expression is given in \cite{Mehta}:
\begin{equation}\label{eigendensity}F_{\beta,n}(\lambda) = C_\beta(n) \exp\left(-\frac{\beta}{2} \sum_{j=1}^n \lambda_j^2 \right) \prod_{j,k \in [1,n]} |\lambda_k - \lambda_j |^{\beta/2},\end{equation}
where $C_{\beta}(n)$ is a normalization constant (its explicit value is written in equation \eqref{densityconstant} below).
In particular one can write $f_{\beta, n}(\e)$ as integral of the joint density for the eigenvalues over the region $((-\infty, -\e)\cup(\e, \infty))^n$; this rather explicit expression allows us to derive the following exact formula. Note that the following theorem holds for arbitrary $\beta>0$, but in our applications we are interested in $\beta=1,2,4$; for these three cases we develop more explicit formulas (Lemmas \ref{b1}, \ref{b2} and \ref{b4}) and the asymptotic stated in Corollary \ref{gapasymptotic}.
\begin{thmi}[The derivative of the gap probability at zero]
\begin{equation}\label{gapexct}f_{\beta, n}'(0)=-2n\frac{C_\beta(n)}{C_\beta(n-1)}\mathbb{E}_{Q\in G_{\beta, n-1}}\left \{|\det(Q)|^\beta\right\}\end{equation}
\end{thmi}
In fact all the quantities in the above equation are classically known (the expectation of the modulus of the determinant is computed in \cite{Mehta} using the Mellin transform; notice that it is for an ensemble of dimension one less than the original one). The asymptotic analysis of these quantities, as needed for Proposition \ref{gapasympt}, is still delicate and is part of the results of this paper.

\subsection*{Related problems}We discuss in this section some interesting related problems. 

To start with, the theorem on the intersection of $k$ quadrics answers question \eqref{q2} above only for Betti numbers away from the middle ones. What is the behavior of these middle Betti numbers remains an open problem (for $k>2$). Here we want to motivate why we believe that their contribution should be of the order $O(n^{(k-1)/2}).$ As we already mentioned, the set $\Sigma_{1,n}$ has singularities of codimension three in $G_{1,n}=\textrm{Sym}(n, \RR)$; for a generic choice of a $k$-dimensional space ($k\geq 3$) $\Sigma_{1,n}\cap W$ is smooth, but when $k$ is bigger singularities are unavoidable. Let us set $\Sigma_{W}^{(1)}=\Sigma_{1,n}\cap W$ and denote by $\Sigma_W^{(r)}$ the set of singular points of $\Sigma_{W}^{(r-1)}$.  In \cite{Lerariocomplexity} the first author proves that for  a generic choice of $W=\textrm{span}\{q_1, \ldots, q_k\}$ the following inequality holds between the Betti numbers of $X=\{q_1=\cdots=q_k=0\}$ and those of $\Sigma_{W}^{(r)}$:
$$b(X)-n\leq \frac{1}{2}\sum_{r\geq 1}b(\Sigma_{W}^{(r)}).$$
Thus in a sense, the homological complexity of $X$ is bounded by the sum of the complexities of the $\Sigma_{W}^{(r)}$. For example in the case $k=2$ we see that $b(X)-n$ is bounded by the number of singular lines in $W\cap \Sigma_{1,n}$ (in this case actually the above formula is almost exact and is the key for the theorem on the intersection of two random quadrics). Since these lines are defined by a homogeneous equation of degree $n$ they are at most $n$; despite this, when we look at their average number we see there are order $O(n^{1/2})$ of them (this can be read from \eqref{matrixpoly}). 

\begin{center}
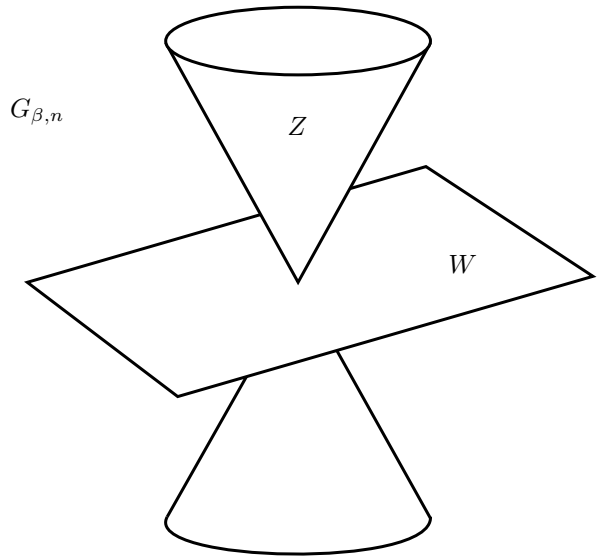
\begin{figure}\label{figrandomplane}
\scalebox{1} 
{
\begin{pspicture}(0,-3.77)(8.441015,3.75)
\psbezier[linewidth=0.04,fillstyle=solid](2.7610157,-3.05)(2.4210157,-3.67)(6.301016,-3.75)(6.2610154,-3.05)
\usefont{T1}{ptm}{m}{n}
\rput(7.1124706,0.335){$W$}
\usefont{T1}{ptm}{m}{n}
\rput(4.5124707,2.355){$Z$}
\usefont{T1}{ptm}{m}{n}
\rput(1.0324707,2.335){$G_{\beta,n}$}
\psline[linewidth=0.04cm](4.5010157,0.03)(2.7610157,-3.05)
\psline[linewidth=0.04cm](4.5210156,0.05)(6.2610154,-3.07)
\pspolygon[linewidth=0.04,fillstyle=solid](0.90101564,0.07)(2.9010155,-1.45)(8.421016,0.15)(6.2010155,1.61)
\rput{-180.0}(9.002031,3.28){\pstriangle[linewidth=0.04,dimen=outer,fillstyle=solid](4.5010157,0.03)(3.56,3.22)}
\psellipse[linewidth=0.04,dimen=outer,fillstyle=solid](4.5010157,3.28)(1.78,0.47)
\usefont{T1}{ptm}{m}{n}
\rput(6.6924706,0.315){$W$}
\usefont{T1}{ptm}{m}{n}
\rput(4.4924707,2.155){$Z$}
\end{pspicture} 
}

\caption{A random plane $W$ in $G_{\beta,n}$; here $Z$ is the set of singular matrices. What is the probability that this plane intersects $Z$ only at the origin? In Theorem \ref{bettirandomtwo} one has to count the number of lines in $Z\cap W.$}
\end{figure}
\end{center}

In the case $k=3$ the situation is fairly more complicated; here we have to compute the average number of components of the random curve:
$$\Sigma_W=\{(\omega_1, \omega_2, \omega_3)\in S^2\,|\, \det(\omega_1Q_1+\omega_2Q_2+\omega_3Q_3)=0\},\quad Q_1, Q_2, Q_3\in \textrm{GOE}(n).$$
A theorem of V. Vinnikov \cite{Vinnikov} states that \emph{every} real algebraic curve arises in this way; thus the above construction gives another possible model for random curves (already introduced in \cite{Lerario2012, LerarioLundberg}): curves arising as random hyperplane sections of $\Sigma_{1,n}.$ 

Gayet and Welschinger's result on random algebraic manifolds states that fixing $n$ and letting the degree going to infinity one gets on average a homological complexity which is of order the square root of the complete intersection in complex projective space \cite{GayetWelschinger4}. Here for $X$ we are performing the opposite limit ($k$ and $d=2$ are fixed and $n\to \infty$) but by analogy we guess that we should get on average the square root of the homological complexity of a complete intersection of $k$ quadrics in the projective space, which is of order $O(n^{k-1}).$ In particular, because of the theorem on $k$ quadrics, we see that this homological complexity should come from middle Betti numbers, which in turn are bounded by the r.h.s. of the above inequality.

Closely related to this is a random version of a problem studied by J. Adams, P. Lax and R. Phillips. In \cite{ALP} the authors studied (topological) restrictions on the dimension of a subspace $W$ of $G_{\beta,n}$ missing $\Sigma_{\beta,n}$ (i.e. $W$ hits the set of singular matrices only at the origin). 
It turns out that the maximal dimension $\rho_\beta(n)$ of such $W$ is related to Radon-Hurwitz numbers (see \cite{ALP}). In particular if $k\leq \rho_\beta(n)$, it is natural to ask the question:
$$\textrm{``What is the probability that a random $k$-dimensional space in $G_{\beta,n}$ misses $\Sigma_{\beta,n}$?''}$$
Notice that in the case $n$ is odd we must hit $\Sigma_{\beta, n}$ at a nonzero point (in fact the determinant is a real polynomial of \emph{odd} degree and must vanish somewhere on $W$). For this reason we introduce the following event:
$$L_\beta(k,n)=\left\{W\simeq \RR^k\subset G_{\beta,n}\,\textrm{such that} \,\max_{Q\in W} \ii(Q)=\left\lfloor \frac{n+1}{2}\right\rfloor\right\}$$
where $\ii(Q)$ denotes the number of positive eigenvalues of the matrix $Q$ from $G_{\beta,n}.$ It is not difficult to show that the probability of missing $\Sigma_{\beta,n}$ for even $n$ equals the probability of $ L_{\beta}(k,n)$. Thus we can ask more generally for the probability of $L_{\beta}(k,n)$ and its asymptotic. 

This problem is related to the above question on the topology of intersection of three random quadrics, as:
$$\PP\{L_{1}(3,n)\}=\textrm{probability that a random determinantal curve of degree $n$ is empty}.$$

\subsection*{Structure of the paper}In Section \ref{sec:eck-you} we prove the generalization of Eckart-Young Theorem for our set of $\beta$-Hermitian matrices. Section \ref{sec:gap} is devoted to the explicit computation of the derivative of the gap probability at zero (Theorem \ref{lemma:joint}); its large $n$ behavior (Corollary \ref{gapasymptotic}) will actually follow from the asymptotic analysis we perform in Section \ref{sec:volume}, where we study the volume of $\Sigma_{\beta, n}$ (Theorem \ref{thm:golden}); in this section we also prove the result on the expected number of zeros of determinant of matrix functions. Section \ref{sec:quadrics} is devoted to random intersection of quadrics; in the first part of the section we recall what we need from algebraic topology, giving a short account of spectral sequences; in the second part we give the proofs of the theorems on average Betti numbers (Theorem \ref{thm:bettiquadrics} and Corollary \ref{cor:interval} for the case of $k$ quadrics and Theorem \ref{bettirandomtwo} for the case of two).

\subsection*{Acknowledgements}
 The authors wish to thank Saugata Basu for his constant support and Peter Sarnak for helpful suggestions.

\section{Eckart-Young theorem}\label{sec:eck-you}

In this section we prove a generalization of the theorem of Eckart and Young \cite{EckartYoung}. In its classical statement it provides the distance (in the Frobenius norm) between an invertible matrix $Q$ and the set $Z$ of matrices with determinant zero:
$$d(Q, Z)=\|Q^{-1}\|^{-1}_{\textrm{op}}.$$
We will be interested in computing the above distance in the metric space $G_{\beta,n}$ (the distance is again induced by the Fobenius norm, but a priori it could be bigger than the above one; a statement for the case of real symmetric matrices already appeared in \cite{H}). Notice that if $Q\in G_{\beta, n}$ is invertible then $\|Q^{-1}\|^{-1}_{\textrm{op}}$ equals the least singular value $\sigma(Q)$ of $Q$. 
\begin{thm}\label{Eck-You}
Let $Q\in G_{\beta,n}$ be invertible and $Z$ be the set of matrices in $G_{\beta,n}$ with determinant zero. Then:
$$d_{G_{\beta,n}}(Q,Z)=\sigma(Q).$$
\end{thm}

\begin{proof}Given $Q\in G_{\beta, n}$ invertible we consider the function:
$$f_Q:X\mapsto \|Q-X\|^2$$
and we look for a minimum on $Z$. We first prove that one such minimum must have rank $n-1.$ In fact let  $\hat{X}$ such that:
$$\textrm{rank}(\hat{X})\leq n-2\quad \textrm{and}\quad \|Q-\hat{X}\|^2=\min_{X\in Z}\|Q-X\|^2 .$$
For $\beta=1,2,4$ let $V_\beta$ be respectively $\RR, \CC$ and $\mathbb{H};$
then for every $v\in V_\beta^n$ of norm one we have $\textrm{rank}\left(\hat{X}-\e v\overline{b}^T\right)\leq n-1$ and:
$$\left\|Q-\hat{X}-\e v\overline{v}^T\right\|^2=\|Q-\hat{X}\|^2-2\e \left\langle Q-\hat{X},v\overline{v}^T\right\rangle+\e^2\geq\|Q-\hat{X}\|^2.$$
In particular $\e^2\geq 2\e \left\langle Q-\hat{X},v\overline{v}^T\right\rangle=2\e \textrm{tr}\left((Q-\hat{X})\overline{v}v^T\right)$ and since this holds for every $v$ of norm one, then $Q=\hat{X}$ which is impossible.\\
Thus we restrict to find minima of $f_Q$ on the \emph{smooth} stratum $Z_1$ of $Z$ where the corank is one: since the stratum is smooth we can use Lagrange multipliers rule. 
Over this stratum we have $(\nabla \textrm{det})_X=\alpha(X)$ (the adjoint matrix of $X$) and $(\nabla f_Q)_X=2(Q-X).$ 
Thus $X\in Z_1$ is a critical point of $f_Q$ if and only if for some $\lambda$ we have $\lambda (Q-X)=\alpha(X).$ In particular multiplying both sides of this equation by $X$ and using the identity $\alpha(X)X=\textrm{det}(X)\mathbbm{1}$ we get:
$$X\in\textrm{Crit}(f_Q|_{Z_1})\quad\textrm{implies}\quad QX=X^2.$$
In particular since $\overline{X}^T=X$ we get:
$$QX=X^2=\left(\overline{X}^T\right)^2=\overline{QX}^T=\overline{X}^T\overline{Q}^T=XQ$$
which says $Q$ and $X$ can be simultaneously diagonalized by $\beta$-unitary operators; since conjugation by such operators is an isometry in $G_{\beta, n}$, then we can assume both $Q$ and $X$ are already diagonal andx satisfy $QX=X^2$; it remains to compute the norm of $Q-X$ to detect minima. If $Q=\textrm{diag}(\lambda_1, \ldots, \lambda_n)$ and $X=(x_1, \ldots, x_n)$, then for every $k=1, \ldots, n$ we have $\lambda_kx_k=x_k^2$ and $\lambda_k, x_k\in \RR.$ Thus:
$$\|Q-X\|^2=\sum_{k=1}^n\lambda_k^2-x_k^2.$$
Now we already know that the matrix $X$ has rank $n-1$ and exactly one of the $x_k$ is zero, say $x_s$. In particular $\|Q-X\|^2=x_s^2$ and the minimum of $f_Q$ is attained when $x_s^2=\sigma(Q)^2.$

\end{proof}

\begin{remark}In fact the proof can be adapted to find critical points of $f_Q$ over a stratum of $Z$ where the corank is bigger, say $r$. Then a corresponding minimum is obtained by setting in the diagonal form of $Q$ the first $r$ singular values to zero. Also note that if $Q$ has multiple eigenvalues we can find several minima.
\end{remark}

\section{Gap probabilities}\label{sec:gap}
In this section we derive the explicit formula for $\displaystyle f'_{\beta,n}(0)=\lim_{\e\to0^+} f'_{\beta, n}(\e)$.
\begin{thm}\label{lemma:joint}

$$f_{\beta, n}'(0)=-2n\frac{C_\beta(n)}{C_\beta(n-1)}\mathbb{E}_{Q\in G_{\beta, n-1}}\left \{|\det(Q)|^\beta\right\} .$$ 

\end{thm}

\begin{proof}

We make use of the important and well-known formula \cite[Ch. 3]{Mehta} for the joint p.d.f. $F_{\beta,n}(\lambda)$ of
the eigenvalues $\lambda = ( \lambda_1, \lambda_2,..,\lambda_n )$ of a random matrix in the $G_{\beta,n}$ ensemble:
$$F_{\beta,n}(\lambda) = C_\beta(n) \exp\left(-\frac{\beta}{2} \sum_{j=1}^n \lambda_j^2 \right) \prod_{j,k \in [1,n]} |\lambda_k - \lambda_j |^{\beta/2},$$
where
\begin{equation}\label{densityconstant}C_\beta(n) = (2\pi)^{-n/2} \beta^{n(n-1)\beta/4 + n/2} \prod_{j=1}^{n} \frac{\Gamma(1+\beta/2)}{\Gamma(1+j\beta/2)} .\end{equation}

By definition, $\f(\e)$ is an integral of $F_\beta(\lambda)$
over the set where all eigenvalues have absolute value at least $\e$, i.e.
\begin{equation}\label{gapderivative} \f(\e) = \PP \{ \sigma(Q) \geq \e \} =  \int_{\left((-\infty, \e)\cup(\e,+\infty)\right)^n} F_{\beta,n}(\lambda) d\lambda, \end{equation}
Next differentiate both sides of this equation with respect to $\e$.
Since $\e$ only appears in the limits of integration (and not in the integrand),
this simply produces boundary integrals; by symmetry under permutations of the variables, these $n$ integrals are equal, so we just multiply the last one by $n$
(here we have used $\hat{\lambda}$ to denote the first $n-1$ entries in $\lambda$):
\begin{equation}\label{star} \f'(\e) = - n \int_{\left((-\infty, -\e)\cup(\e,+\infty)\right)^{n-1}} \left(F_{\beta,n}(\hat{\lambda}, -\e)+F_{\beta, n}( \hat{\lambda},\e) \right)d\hat{\lambda}.\end{equation}
For example, let us consider the case $n=2$; we have: 
\begin{align*}f_{\beta, 2}(\e)&=\left(\int_{-\infty}^{-\e}\int_{-\infty}^{-\e}+\int_{-\infty}^{-\e}\int_{\e}^\infty+\int_{\e}^{\infty}\int_{\e}^\infty+\int_{\e}^{\infty}\int_{-\infty}^{-\e}\right)F_{\beta, 2}(\lambda_1,\lambda_2)d\lambda_1d\lambda_2\\
&=\left(\int_{-\infty}^{-\e}\int_{-\infty}^{-\e}-\int_{-\infty}^{-\e}\int_{\infty}^\e+\int_{\infty}^{\e}\int_{\infty}^\e-\int_{\infty}^{\e}\int_{-\infty}^{-\e}\right)F_{\beta, 2}(\lambda_1,\lambda_2)d\lambda_1d\lambda_2
\end{align*}
Thus taking derivative  the fundamental Theorem of calculus implies:\begin{small}
 \begin{align*}
 f'_{\beta, 2}(\e)=&-\int_{-\infty}^{-\e} F_{\beta, 2}(-\e, \lambda_2)d\lambda_2-\int_{-\infty}^{-\e} F_{\beta, 2}(\lambda_1, -\e)d\lambda_1+\int_{\infty}^\e F_{\beta, 2}(-\e, \lambda_2)d\lambda_2-\int_{-\infty}^{-\e} F_{\beta, 2}(\lambda_1, \e)d\lambda_2\\
 &+\int_{\infty}^\e F_{\beta, 2}(\e, \lambda_2)d\lambda_2+\int_{\infty}^\e F_{\beta, 2}( \lambda_1, \e)d\lambda_1-\int_{-\infty}^{-\e}F_{\beta, 2}(\e, \lambda_2)\lambda_2+\int_{\infty}^\e F_{\beta, 2}(\lambda_1, -\e)d\lambda_1\\
 =&-\int_{-\infty}^{-\e} F_{\beta, 2}(-\e, \lambda_2)d\lambda_2-\int_{-\infty}^{-\e} F_{\beta, 2}(\lambda_1, -\e)d\lambda_1-\int_{\e}^\infty F_{\beta, 2}(-\e, \lambda_2)d\lambda_2-\int_{-\infty}^{-\e} F_{\beta, 2}(\lambda_1, \e)d\lambda_2\\
 &-\int_{\e}^\infty F_{\beta, 2}(\e, \lambda_2)d\lambda_2-\int_{\e}^\infty F_{\beta, 2}( \lambda_1, \e)d\lambda_1-\int_{-\infty}^{-\e}F_{\beta, 2}(\e, \lambda_2)\lambda_2-\int_{\e}^\infty F_{\beta, 2}(\lambda_1, -\e)d\lambda_1=(*)
 \end{align*}
 \end{small}
 Collecting together and using the invariance under permutations of the variables of $F_{\beta,2}$, we can rewrite the above expression as:
 \begin{small}
 \begin{align*}
(*)=& \int_{(-\infty, -\e)\cup(\e, +\infty)}F_{\beta, 2}(-\e, \lambda_2)+F_{\beta, 2}(\e, \lambda_2)d\lambda_2-\int_{(-\infty, -\e)\cup(\e, +\infty)}F_{\beta, 2}(\lambda_1, -\e)+F_{\beta, 2}( \lambda_1, \e)d\lambda_1\\
 &=-2\int_{(-\infty, -\e)\cup(\e, +\infty)}F_{\beta, 2}(-\e, \lambda_2)+F_{\beta, 2}(\e, \lambda_2)d\lambda_2
 \end{align*}
\end{small}
Taking the limit $\e \ra 0$ in the general case \eqref{star}, the dominated convergence Theorem gives:
\begin{equation}
\label{eq:faces}
\lim_{\e \ra 0} \f'(\e) = -2 n \int_{\mathbb{R}^{n-1}} F_{\beta,n}(\hat{\lambda},0) d\hat{\lambda} .
\end{equation}
The integrand $F_{\beta,n}(\hat{\lambda},0)$ can be rearranged in an interesting way
that expresses it in terms of the density $F_{\beta,n-1}(\hat{\lambda})$
and the determinant $|\lambda_1 \cdot \lambda_2 \cdots \lambda_{n-1}|$:

\begin{align*}
F_{\beta,n}(\hat{\lambda},0) &= C_\beta(n) \exp\left(-\frac{\beta}{2} \sum_{j=1}^{n-1} \lambda_j^2 \right) \left\{ \prod_{j,k \in [1,n]}  |\lambda_k - \lambda_j |^{\beta/2} \right\}_{\lambda_n=0} \\
&= C_\beta(n) \exp\left(-\frac{\beta}{2} \sum_{j=1}^{n-1} \lambda_j^2 \right) 
\left\{ \prod_{j,k \in [1,n-1]}  |\lambda_k - \lambda_j |^{\beta/2} \right\} |\lambda_1 \cdot \lambda_2 \cdots \lambda_{n-1}|^{\beta} \\
&= \frac{C_\beta(n)}{C_\beta(n-1)} |\lambda_1 \cdot \lambda_2 \cdots \lambda_{n-1}|^{\beta} F_{\beta,n-1}(\hat{\lambda})
\end{align*}

Substituting this into (\ref{eq:faces}) finally leads to an expression for 
$\lim_{\e \ra 0} f_{\beta,n}'(\e)$ in terms of the Mellin transform:
\begin{align*}
\lim_{\e \ra 0} \f'(\e) &= -2n \frac{C_\beta(n)}{C_\beta(n-1)}\int_{\mathbb{R}^{n-1}} |\lambda_1 \cdot \lambda_2 \cdots \lambda_{n-1}|^{\beta} F_{\beta,n-1}(\hat{\lambda})  d\hat{\lambda} \\
&= -2n\frac{C_\beta(n)}{C_\beta(n-1)}\mathbb{E}_{Q\in G_{\beta, n-1}}\left \{|\det(Q)|^\beta\right\}
\end{align*}

\end{proof}

\section{The intrinsic volume of $\Sigma_{\beta,n}$}\label{sec:volume}

In this section we compute the intrinsic volume (induced by the Frobenius norm) of the set:
$$\Sigma_{\beta, n}=\{Q\in G_{\beta, n}\, : \|Q\|^2=1 \textrm{ and } \det(Q)=0\}.$$

Let us recall that the Mellin transform $\Mel(\beta,s)$ \cite{Mehta}
provides the moments of the determinant of a random matrix from the $\beta$-ensemble.
Namely,
\begin{equation}\label{eq:mellin}\Mellin(\beta,s) = \frac{1}{2} \EE_{\beta,n} |\det|^{s-1} = \frac{1}{2} \EE_{\beta,n} |\lambda_1 \cdot \lambda_2 \cdots \lambda_n|^{s-1},\end{equation}
where we have used the notation $\EE_{\beta,n}$ to emphasize that the expectation is taken using the probability density associated to $G_{\beta, n}$.
In terms of this, we state a precise formula for the volume  of $\Sigma_{\beta, n}$.

\begin{thm}\label{thm:golden}
For $\beta=1,2,4$:
$$|\Sigma_{\beta, n}| = 2n  \sqrt{2\pi}\frac{C_\beta(n)}{C_\beta(n-1)} \Mel(\beta,\beta+1)\cdot |S^{N_\beta-2}| ,$$
\end{thm}
where $N_{\beta,n}=\dim G_{\beta, n}=n+\frac{1}{2}n(n-1)\beta$, and $C_\beta(n)$ is defined in (\ref{densityconstant}).
\begin{remark}
Plugging in the exact values for the normalization constants we have: 
$$ \frac{C_\beta(n)}{C_\beta(n-1)} = \frac{\beta^{(n-1)\beta/2+1/2}\Gamma(1+\beta/2)}{\sqrt{2\pi} \Gamma(1+\beta n/2)}$$
which in turn implies:
$$|\Sigma_{\beta, n}|=2n \beta^{\frac{n\beta -\beta +1}{2}}\frac{\Gamma(1+\beta/2)}{\Gamma(1+\beta n/2)}\Mel(\beta, \beta+1)\cdot |S^{N_\beta-2}|.$$
\end{remark}

We prove this using Theorem \ref{Eck-You} along with several lemmas. We first reduce the problem to some Random Matrix Theory computations. To avoid cumbersome notation, let us suppress the dependence on $\beta$ and $n$ for the next two statements.

\begin{prop}\label{prop:gap} $$|\Sigma|=|S^{N-1}|\cdot\lim_{\e\to 0}\frac{\PP\{\sigma(Q)\leq \e \|Q\|\}}{2\e}$$
\end{prop}
\begin{proof}
Since $\Sigma$ is an algebraic subset of the sphere $S^{N-1}$ of codimension one, then its intrinsic volume is computed by: $$|\Sigma|=\lim_{\e\to 0} \frac{|\mathcal{U}_{S^{N-1}}(\Sigma, \e)|}{2\e},$$
where $\mathcal{U}_{S^{N-1}}(\Sigma, \e)$ is an $\e$-tube aorund $\Sigma$ in $S^{N-1}$ (i.e. the set of points in $S^{N-1}$ at distance less than $\e$ from $\Sigma$). 

Consider now the functions $v(\e)=|\mathcal{U}_{S^{N-1}}(\Sigma,\e)|$ and $\hat{v}(\e)=|\mathcal{U}_G(Z, \e)\cap S^{N-1}|$ (where $\mathcal{U}_G(Z, \e)$ is an $\e$-tube of $Z$ in $G$). We will prove that these functions have the same derivative at zero. First notice that if $d_{S^{N-1}}(X, \Sigma)\leq \e$ then also $d_{G}(X, Z)\leq \e$; hence:
\begin{equation}\label{vol1}\mathcal{U}_{S^{N-1}}(\Sigma, \e)\subset \mathcal{U}_G(Z, \e)\cap S^{N-1}.\end{equation}
Assume now that $d_{M}(X, Z)\leq \e$ for $X\in S^{N-1}.$ Then the geodesic joining $X$ to $\Sigma$ is an ``arc'' on the sphere and by the triangle inequality $d_{S^{N-1}}(X, \Sigma)\leq d_G(X, Z) + 1-\cos (d_G(X, Z))\leq \e +\e^2$ (the last inequality for $\e$ small enough). In particular we get the inclusion:
\begin{equation}\label{vol2}\mathcal{U}_G(Z, \e)\cap S^{N-1}\subset \mathcal{U}_{S^{N-1}}(\Sigma, \e+\e^2).\end{equation}
Combining (\ref{vol1}) and (\ref{vol2}) gives $v(\e) \leq \hat{v}(\e) \leq v(\e+\e^2),$ which in turn implies $\lim_{\e\to 0}\frac{v(\e)}{2\e}=\lim_{\e\to0}\frac{\hat{v}(\e)}{2\e}.$
In particular this implies that $|\Sigma|$ is also computed by:
\begin{equation}\label{eqvolp}|\Sigma|=\lim_{\e\to 0} \frac{|\mathcal{U}_G(Z, \e)\cap S^{N-1}|}{2\e}.\end{equation}
We apply now Theorem \ref{Eck-You} getting that:
$$\mathcal{U}_G(Z, \e)\cap S^{N-1}=\{\sigma(Q)\leq \e\}\cap S.$$
Since the probability distribution on the ensemble $G$ is uniform on the unit sphere, then the volume of $\{\sigma(Q)\leq \e\}\cap S$ equals $|S^{N-1}|$ times the probability of the cone generated by $\{\sigma(Q)\leq \e\}\cap S^{N-1}$; in other words:
$$|\mathcal{U}_G(Z,\e)\cap S^{N-1}|=|S^{N-1}|\cdot\mathbb{P}\{\sigma(Q)\leq \e \|Q\|\}.$$
\end{proof}

In the following, we have the least singular value in mind for the function $\sigma(Q)$,
but we state the Lemma in more generality.

\begin{lemma}\label{lemma:cylcone}
Fix $N$ and for $Q \in \RR^N$ suppose $\sigma(Q)$ is a continuous positive function that is homogeneous of degree one,
so $\sigma(Q) = \| Q \| \sigma(Q / \| Q \|) $.
Define:
$$f(\e) = \PP \{ \sigma(Q) \geq \e \}\quad
\textrm{and} \quad g(\e) = \PP \{ \sigma(Q) \geq \e \| Q \| \}. $$
Then,
$$\lim_{\e \ra 0} - g'(\e) = \left( \lim_{\e \ra 0} - f'(\e) \right) \frac{\sqrt{2} \Gamma(\frac{N}{2})}{\Gamma(\frac{N-1}{2})} .$$ 
\end{lemma}

\begin{proof}
First we establish the equation:
\begin{equation}\label{eq:f}
f(\e) = \frac{\text{Vol}(S^{N-1})}{(2 \pi)^{N/2}} \int_0^\infty g(\e / r) r^{N-1} e^{-\frac{r^2}{2}} dr.
\end{equation}

Starting from the definition for $f$, we have:

\begin{align*}
f(\e) &= \frac{1}{(2 \pi)^{N/2}}  \int_0^\infty \int_{S^{N-1}} \chi_{\{ \sigma(Q) \geq \e \}}  r^{N-1} e^{-\frac{r^2}{2}} d\theta dr \\
&= \frac{1}{(2 \pi)^{N/2}}  \int_0^\infty \underbrace{\int_{S^{N-1}} \chi_{\{ \sigma(Q) \geq \e \}} d\theta}_{\text{Vol}(S^{N-1}) g(\e/r)} r^{N-1} e^{-\frac{r^2}{2}} dr \\
&=\frac{\text{Vol}(S^{N-1})}{(2 \pi)^{N/2}}  \int_0^\infty  g(\e/r) r^{N-1} e^{-\frac{r^2}{2}} dr
.
\end{align*}

This proves (\ref{eq:f}).
Differentiating (\ref{eq:f}), we get:
\begin{equation}\label{eq:fprime}
f'(\e) = \frac{\text{Vol}(S^{N-1})}{(2 \pi)^{N/2}} \int_0^\infty g'(\e / r) r^{N-2} e^{-\frac{r^2}{2}} dr.
\end{equation}

Next we take the limit $\e \ra 0$ and apply the dominated convergence theorem. 
Note that for each $\e$, $g'(\e/r)$ as a function of $r$ is continuous and has a limit as $\e \ra 0$ and $\e \ra +\infty$.
This implies that $|g'(\e/r)|$ is uniformly bounded by say $M$.
Thus, for all $\e>0$, the integrands in (\ref{eq:fprime})
are all dominated by the integrable function $M r^{N-2} e^{-\frac{r^2}{2}}$.
This justifies the following:

\begin{align*}
\lim_{\e \ra 0} \int_0^\infty g'(\e / r) r^{N-2} e^{-\frac{r^2}{2}} dr 
&= \int_0^\infty \lim_{\e \ra 0} g'(\e / r) r^{N-2} e^{-\frac{r^2}{2}} dr \\
&= \int_0^\infty \lim_{\e \ra 0} g'(\e) r^{N-2} e^{-\frac{r^2}{2}} dr \\
&=  \lim_{\e \ra 0} g'(\e) \int_0^\infty r^{N-2} e^{-\frac{r^2}{2}} dr .
\end{align*}

Applying this while taking the limit $\e \ra 0$ in (\ref{eq:fprime}) yields the statement in the Lemma:
\begin{align*}
\lim_{\e \ra 0} f'(\e) &=   \left( \lim_{\e \ra 0} g'(\e) \right) \underbrace{ \frac{\text{Vol}(S^{N-1})}{(2 \pi)^{N/2}} \int_0^\infty r^{N-2} e^{-\frac{r^2}{2}} dr}_{\frac{\Gamma(\frac{N-1}{2})}{\sqrt{2} \Gamma(\frac{N}{2})}} \\
&= \left( \lim_{\e \ra 0} g'(\e) \right) \frac{\Gamma(\frac{N-1}{2})}{\sqrt{2} \Gamma(\frac{N}{2})}.
\end{align*}\end{proof}

We are finally in the position to give the proof of Theorem \ref{thm:golden}, which in fact is just given by the following chain of equalities:
\begin{align*}
|\Sigma_{\beta, n}| &= |S^{N_\beta-1}|\cdot \lim_{\e\to 0}\frac{1-\PP\{\sigma(Q)\geq \e \|Q\|\}}{2\e}=|S^{N_\beta-1}|\cdot \lim_{\e\to 0}\frac{1-g(\e)}{2\e}\quad(\textrm{by Proposition \ref{prop:gap}})\\
&=|S^{N_\beta-1}|\cdot \lim_{\e\to 0}\frac{-g'(\e)}{2}=|S^{N_\beta-1}|\frac{\sqrt{2}\Gamma(\frac{N_\beta}{2})}{\Gamma(\frac{N_\beta-1}{2})}\cdot \frac{1}{2}\left( \lim_{\e \ra 0} - f'(\e) \right)\quad(\textrm{by Lemma \ref{lemma:cylcone}})\\
&= \underbrace{2n|S^{N_\beta-1}|\frac{\sqrt{2}\Gamma(\frac{N_\beta}{2})}{\Gamma(\frac{N_\beta-1}{2})}}_{2n\sqrt{2\pi}|S^{N_\beta-2}|} \frac{C_\beta(n)}{C_\beta(n-1)}\Mel(\beta,\beta+1) \quad(\textrm{by Theorem \ref{lemma:joint} and \eqref{eq:mellin}}).
\end{align*}

\subsection{Explicit expressions for the constants and asymptotic analysis of Theorem \ref{thm:golden}}

The following lemma is a combination of Prop. 7 and Cor. 3 from \cite{GayetWelschinger3} and gives the exact formula for $\Mel(1,2)$ together with its asymptotic behavior.

\begin{lemma}\label{b1}
\begin{equation*}
\Mel(1,2) = \frac{1}{2}\left\{
\begin{array}{cc}
\frac{2\sqrt{2}}{\pi}\Gamma\left(\frac{n+1}{2}\right) & \textrm{for even $n$},\\
(-1)^m\frac{(n-1)!}{m!2^{n-1}}+(-1)^{m-1}\frac{4\sqrt{2}(n-1)!}{\sqrt{\pi}m!2^{n-1}}\sum_{k=0}^{m-1}(-1)^k\frac{\Gamma(k+3/2)}{k!} & \textrm{for odd $n=2m+1$}\\
\end{array} \right.
\end{equation*}
Moreover as $n$ goes to infinity (regardless its parity): $$\Mel(1,2)\sim \frac{\sqrt{2}}{\pi}\Gamma\left(\frac{n+1}{2}\right).$$
\end{lemma}

As for the other two cases we have the following lemmas.

\begin{lemma}\label{b2}
\begin{equation*}
\Mel(2,3) = \left\{
\begin{array}{cc}
\frac{1}{\pi}\Gamma\left(\frac{n+1}{2}\right)^2 & \textrm{for even $n$},\\
\frac{n}{2\pi}\Gamma\left(\frac{n}{2}\right)^2 & \textrm{for odd $n=2m+1$}\\
\end{array} \right.
\end{equation*}
Moreover as $n$ goes to infinity (regardless its parity): $$\Mel(2,3)\sim \frac{1}{\pi}\Gamma\left(\frac{n+1}{2}\right)^2.$$
\end{lemma}
\begin{proof}
We start by recalling the following formula from \cite{Mehta}:
$$\Mel(2,3)=\frac{1}{2}\prod_{j=1}^{n-1}\frac{\Gamma(3/2+\lfloor j/2\rfloor)}{\Gamma(1/2+\lfloor j/2\rfloor)}.$$
Using the identity $\Gamma(z+1)=z\Gamma(z)$ in the above formula with $z=1/2+\lfloor j/2\rfloor,$ we can rewrite it as:
\begin{align*}\Mel(2,3)&=\frac{1}{2}\prod_{j=1}^{n-1}\left(\frac{1}{2}+\left\lfloor\frac{j}{2}\right\rfloor\right)\\&=\frac{1}{2}\left(\prod_{\textrm{$1\leq j\leq n-1$, $j$ even}}\frac{j+1}{2}\right)\cdot\left(\prod_{\textrm{$1\leq j\leq n-1$, $j$ odd}}\frac{j}{2}\right)\\
&=\frac{1}{2^{n}}\left(\prod_{2\leq k\leq n,\textrm{ $k$ odd}}k\right)\cdot \left(\prod_{1\leq k\leq n,\textrm{ $k$ odd}}k\right)=\frac{c_n}{2^n}\prod_{1\leq k\leq n,\textrm{ $k$ odd}}k^2,\end{align*}
where $c_n=1$ for even $n$ and $n$ for odd ones. Thus if $n=2m$ is even, we have:
\begin{align*}\Mel(2,3)&=\frac{1}{2^n}(2m-1)!!^2\\
&=\frac{2^{2m-n}}{\pi}\Gamma(m+1/2)^2=\frac{1}{\pi}\Gamma\left(\frac{n+1}{2}\right)^2,
\end{align*}
where in the last line we have used the identity $\Gamma(m+1/2)=\sqrt{\pi}\frac{(2m-1)!!}{2^m}.$
In the case $n=2m+1$ is odd, recalling the value $c_{n\textrm{ odd}}=n$, have:
\begin{align*}\Mel(2,3)&=n\frac{1}{2^n}(1\cdot2\cdots(2m-1))^2=n\frac{\Gamma(m+1/2)^2}{2\pi}\\
&=\frac{n}{2\pi}\Gamma\left(\frac{n}{2}\right)^2.\end{align*}
The asymptotics are a simple application of Stirling's formula.

\end{proof}

\begin{lemma}\label{b4}
$$\Mel(4,5)=\frac{4^{-n+1}}{\pi} \Gamma \left( n+ \frac{1}{2} \right)^2 2 H_n(-1) ,$$
where $H_n$ is a hypergeometric function such that $2 H_n(-1) = 1 + o(1)$ as $n \ra \infty$. In particular as $n$ goes to infinity:
$$\Mel(4,5)\sim \frac{1}{4^{n-1}\pi}\Gamma\left(n+\frac{1}{2}\right)^2.$$
\end{lemma}

\begin{proof}
We start by recalling equation (26.3.10) from \cite{Mehta}:
\begin{align*}\Mel(4,5)&=\frac{1}{2^{2n-1}}\prod_{j=0}^{n-2}\frac{\Gamma(j+5/2)}{\Gamma(j+3/2)}\sum_{k=0}^{n-1}{n-1\choose k}(1)_{k}(3/2)_{n-1-k}\\
&=\frac{1}{2^{2n-1}}\left(\prod_{j=0}^{n-2}\frac{\Gamma(j+5/2)}{\Gamma(j+3/2)}\right)\cdot \frac{2}{\pi}\Gamma \left( n+1/2 \right) \pFq{2}{1} \left( 1,1-n,\frac{1}{2}-n,-1 \right).\end{align*}
In the above line $_2F_1$ denotes the  hypergeometric function; let us set 
$$H_{n}(-1)= \pFq{2}{1} \left( 1,1-n,\frac{1}{2}-n,-1 \right).$$
With this notation we have:
\begin{align*}\Mel(4,5)&=\frac{2 H_{n}(-1)}{\sqrt{\pi}2^{2n-1}}\Gamma(n+1/2)\prod_{j=0}^{n-2}\frac{\Gamma(j+5/2)}{\Gamma(j+3/2)}\\
&=\frac{2 H_{n}(-1)}{\sqrt{\pi}2^{2n-1}} \Gamma(n+1/2)\prod_{j=0}^{n-2}(j+3/2),
\end{align*}
where again the last line we have used $\Gamma(z+1)=z\Gamma(z)$ for $z=j+3/2.$ Recalling also the identity $\prod_{j=0}^{n-2}(j+3/2)=\frac{2}{\pi}\Gamma(n+1/2),$ we finally get:
$$\Mel(4,5)=\frac{2 H_{n}(-1)}{\pi 4^{n-1}}\Gamma\left(n+\frac{1}{2}\right)^2.$$
It remains to prove the limit $2H_n(-1)\to 1$.
First we use the Pfaff transformation:
$$\pFq{2}{1}(a,b;c;z) = (1-z)^{-a} \cdot  \pFq{2}{1} \left( a,b-c;c;\frac{z}{z-1} \right). $$
In our case:
$$2 \cdot \pFq{2}{1} \left( 1,-n;-n-\frac{1}{2};-1 \right) = \pFq{2}{1} \left( 1,-\frac{1}{2};-n-\frac{1}{2};\frac{1}{2} \right).$$
Now we use the series definition in terms of the Pockhammer symbol:
$$\pFq{2}{1}\left( 1,-\frac{1}{2};-n-\frac{1}{2};\frac{1}{2} \right) = \sum_{k=0}^{\infty} \frac{(-1/2)_k}{(-n-1/2)_k} (1/2)^k.$$
We need to show that the right hand side $\rightarrow 1$. We have
\begin{align}\label{eq:hyp}
\sum_{k=0}^{\infty} \frac{(-1/2)_k}{(-n-1/2)_k} (1/2)^k &= 1 + \frac{1}{4(n+1/2)}\sum_{k=1}^{\infty} \frac{(1/2)_{k-1}}{(-n+1/2)_{k-1}} (1/2)^{k-1}.
\end{align}
We use the rough bound:
\begin{align*}
\left| \sum_{k=0}^{\infty} \frac{(1/2)_{k}}{(-n+1/2)_{k}} (1/2)^{k} \right| &\leq \sum_{k=0}^{n} (1/2)^{k} + \frac{1}{|(-n+1/2)_n|} \sum_{k=n+1}^{\infty} \frac{(1/2)_{k}}{(1/2)_{k-n}} (1/2)^{k}, \\
&\leq 2 + \frac{(1/2)^n}{|(-n+1/2)_n|} \sum_{k=n+1}^{\infty} \frac{k!}{(k-n)!} (1/2)^{k-n} \\
&= 2+ \frac{n!}{|(-n+1/2)_n|}(1/2)^n \frac{F^{(n)}(1/2)}{n!} , 
\end{align*}
where
$$F(z) = \frac{1}{1-z} .$$
We have
$$F^{(n)}(1/2)/n! = O(1) \cdot 2^n .$$
Applying this along with Stirling's approximation:
$$2+ \frac{n!}{|(-n+1/2)_n|}(1/2)^n \frac{F^{(n)}(1/2)}{n!} = 2 + \frac{n!}{|(-n+1/2)_n|} O(1) = o(n).$$
This shows that (\ref{eq:hyp}) equals $1 + \frac{1}{4(n+1/2)} o(n) = 1 + o(1)$, as desired.
\end{proof}
As a corollary we get the following asymptotic for the volume of $\Sigma_{\beta,n}$.
\begin{cor}\label{cor:volume}
For each $\beta=1,2,4$ we have:
$$\frac{|\Sigma_{\beta,n} |}{|S^{N_\beta - 2}|}\sim \frac{2}{\sqrt{\pi}}n^{\frac{1}{2}}.$$
\end{cor}

\begin{remark}In our main application of this asymptotic (Thm. \ref{bettirandomtwo}) we will need, for $\beta=1$, a more precise error bound:
$$\frac{|\Sigma_{1,n} |}{|S^{N_1 - 2}|} = \frac{2}{\sqrt{\pi}}n^{\frac{1}{2}} + O(1). $$
\end{remark}
\begin{proof}Recall from Theorem \ref{thm:golden} (and the remark below it) that:
$$ \frac{|\Sigma_{\beta,n} |}{|S^{N_\beta - 2}|} =2n \beta^{\frac{n\beta -\beta +1}{2}}\frac{\Gamma(1+\beta/2)}{\Gamma(1+\beta n/2)}\Mel(\beta, \beta+1).$$
The result follows applying Stirling's approximation to 
 the asymptotic for $\Mel(\beta, \beta+1)$ given in Lemma \ref{b1}, \ref{b2}, \ref{b4}.
 
  The error bound stated in the Remark follows immediately from the error in Stirling's approximation for $n$ even.
  For $n = 2m+1$ odd, reading the proof of Lemma \ref{b1} which was given in \cite[Cor. 3]{GayetWelschinger3}, one can conclude that:
  \begin{equation}\label{eq:GayWel}
  \Mel(1,2) = \frac{(n-1)!}{m! 2^{n-1}} \left( (-1)^m + \frac{4\sqrt{2}}{\sqrt{\pi}} S_m \right),
  \end{equation}
  where
  $$S_m = \frac{1}{2}\sum_{j=0}^{m/2-1} \frac{\Gamma(2j + 2 - 1/2)}{ \Gamma(2j + 2)},$$
  for $m$ even, and
  $$S_m = \frac{\Gamma(m+1/2)}{\Gamma(m)} - \frac{1}{2}\sum_{j=0}^{(m-1)/2-1} \frac{\Gamma(2j + 2 - 1/2)}{ \Gamma(2j + 2)},$$
  for $m$ odd.
  Using the asymptotic \cite{Olver}
  $$\frac{\Gamma(z+a)}{\Gamma(z+b)} = z^{a-b}(1+ O(1/z)),$$
  along with an integral estimate for the sum we have (regardless of the parity of $m$):
  $$S_m = \frac{\sqrt{m}}{2} + O(m^{-1/2}).$$
  
  Applying this to (\ref{eq:GayWel}) gives:
  $$ \Mel(1,2) = \frac{(n-1)!}{m! 2^{n-1}} \left( (-1)^m + \frac{4\sqrt{2}}{\sqrt{\pi}} S_m \right) = \frac{\sqrt{2}}{\pi} \Gamma \left( \frac{n+1}{2} \right) \left( 1 + O(n^{-1/2}) \right).$$
  Using Stirling's approximation for
  $$ \frac{C_1(n)}{C_1(n-1)} = \frac{\Gamma(1+1/2)}{\sqrt{2\pi} \Gamma(1+n/2)}, $$
  and plugging this into the exact formula gives:
  $$\frac{|\Sigma_{1,n} |}{|S^{N_1 - 2}|} = \frac{2}{\sqrt{\pi}}n^{\frac{1}{2}}(1 + O(n^{-1/2})).$$

\end{proof}
The asymptotic of Theorem \ref{lemma:joint} follows again from Lemma \ref{b1}, Lemma \ref{b2} and Lemma \ref{b4}.
\begin{cor}\label{gapasymptotic}The following asymptotic holds for the derivative at zero of the gap probability:
$$f'_{\beta, n}(0)\sim -\frac{2\sqrt{2}}{\pi}n^{1/2}.$$
\end{cor}

\subsection{Zeros of the determinant of a symmetric matrix polynomial}

Inspired by \cite[Thm. 6.1]{EdelmanKostlan95}, the next result counts the average number of zeros of the determinant of a \emph{matrix polynomial}
$$A_0 + t A_1 + t^2 A_2 + ... + t^k A_k,$$
with $A_i$ random matrices from the $\beta$-ensemble (more generally, we consider a combination of differentiable functions multiplied by random matrices in the $\beta$-ensemble).

\begin{cor}
Let $f_0(t),f_1(t),..,f_k(t)$ be any collection of differentiable functions and $A_0,A_1,..,A_k$ random matrices in the $\beta$-ensemble with $\beta=1,2,4$.
Let
$$\EE | \{ \det( A_0 f_0(t) + A_1 f_1(t) + .. + A_k f_k(t) ) = 0 \} | = \alpha_n .$$
Then we have
$$\frac{\alpha_n}{\alpha_1} \sim \frac{2}{\sqrt{\pi}} n^{1/2},$$
where $\alpha_1$ is $1 / \pi$ times the length of the curve $\hat{\gamma}(t) \subset S^k$ obtained from projecting the curve $ \gamma(t) = (f_0(t),f_1(t),..,f_k(t))$
from $\RR^{k+1}$ to $S^k$.
\end{cor}

\begin{proof}
 The proof is the same as in \cite{EdelmanKostlan95} and uses the integral geometry formula to compute:
  $$\alpha_n = \frac{1}{\pi} \frac{|\Sigma_{\beta,n} | |N|}{|S^{N_\beta - 2}|} , $$
 where $|N| = \pi \alpha_1$ is the length of $\hat{\gamma}$,
 and $|\Sigma_{\beta,n}|$ is the intrinsic volume of the the set of singular matrices $\Sigma_{\beta,n}$ in $G_{\beta,n}$.
 In particular: 
 $$\frac{\alpha_n}{\alpha_1} \sim \frac{2}{\sqrt{\pi}} n^{1/2} .$$
 Note that in \cite{EdelmanKostlan95}, instead of $\Sigma_{\beta,n}$, there appeared $M$, the set of singular $n \times n$ real matrices. 
 \end{proof}

\begin{remark} In the case that $f_j(t) = t^j$ as in the above example, $\alpha_1 \sim 2 \log k$.\end{remark}

\section{Intersections of random quadrics}\label{sec:quadrics}
We consider in this section the problem of computing the average Betti numbers of a random intersection of $k$ quadrics in $\RP^{n-1}$. This problem has already been addressed by the first author in \cite{Lerario2012}, where the case $k=1,2$ was studied. 

To make the setting more precise, we consider quadratic forms $q_1, \ldots, q_k$ on $\mathbb{R}^n$ (i.e. homogeneous polynomials of degree two in $n$ variables). Each $q_i$ defines a (possibly empty) quadric hypersurface in the projective space $\RP^{n-1}$ and we consider the set $X$ obtained by intersecting together these hypersurfaces:
$$X=\{[x]\in \RP^{n-1}\,|\, q_1(x)=\cdots=q_k(x)=0\}.$$

Notice that, once a scalar product on $\RR^n$ has been fixed, a quadratic form $q$ determines a unique symmetric matrix $Q$ given by the equation:
$$q(x)=\langle x, Q x\rangle\quad \textrm{for all $x\in \RR^n$}.$$

In the random setting, it is assumed that the quadratic forms $q_i$ are independent and \emph{Kostlan} distributed (see \cite{Lerario2012}); this is equivalent to require that:
$$Q_i\in \textrm{GOE}(n), \quad i=1, \ldots, k.$$

The first result we will prove is the following, which is valid for $i\in\{0, \ldots, \dim(X)\}$ sufficiently bounded away from $\dim(X)/2$ (the ``middle'' Betti numbers).
\begin{thm}\label{thm:bettiquadrics}
Let $X$ be a random intersection of $k$ quadrics in $\RP^{n-1}$. For every $0<\alpha<1$ and $M>0$, if $|i-\frac{n}{2}|\geq n^{\alpha}.$ then:
$$\EE b_i(X) =1+O(n^{-M}).$$
\end{thm}

Note that the convergence in the above range will be uniform. In particular we get the following corollary: since the error term in the above statement is of the order $O(n^{-M})$ \emph{for all} $M>0$, then an accumulation of at most $n$ such error terms has the same property.

\begin{cor}\label{cor:interval}
For every $M>0$ and every open set $J\subset [0,\frac{1}{2})\cup(1/2, 1]$ we have: $$\sum_{j\in nJ}\EE b_i(X)=|\mathbb{N} \cap n J|+O(n^{-M}).$$
\end{cor}

Next we prove the theorem on the intersection of two quadrics.
\begin{thm}\label{bettirandomtwo}For $k=2$ the following formula holds:

$$\mathbb{E}b(X) = n + \frac{2}{\sqrt{\pi}}n^{1/2} + O(n^c) \quad \textrm{for any $c>0$}.$$
\end{thm}

Before proving the theorems, we need a way for computing Betti numbers of $X$ in a \emph{deterministic} setting.

\subsection{Betti numbers of intersections of quadrics: a deterministic approach}The cohomology of intersections of real quadrics is studied in \cite{Agrachev, AgrachevLerario, Lerariotwo, Lerariocomplexity}. The main ingredient is a \emph{cohomology spectral sequence} converging to the homology of $X$; since we will only need part of this machinery, here we present a simplification of the theory adapted to our needs.

Given the quadratic forms $q_1, \ldots, q_k$ and the corresponding symmetric matrices $Q_1, \ldots, Q_k$, we consider their span:
$$W=\textrm{span}\{Q_1, \ldots, Q_k\}\subset \textrm{Sym}(n, \RR).$$
For every symmetric matrix $Q$ we denote by $\ii(Q)$ the number of its positive eigenvalues (usually called the \emph{index} of $Q$); we set
$$\mu=\max_{Q\in W}\ii(Q)\quad \textrm{and}\quad \nu=\min_{Q\in W\backslash \{0\}}\ii(Q).$$
Notice that for the generic choice of $q_1, \ldots, q_k$ the span $W$ is $k$-dimensional.

In order to study the topology of $X$ we will need to consider also, for every $j\geq 0$ the following subset of the sphere $S^{k-1}$:
$$\Omega^{j}=\{\omega \in S^{k-1}\,|\, \ii(\omega Q)\geq j\}$$
where the notation $\omega Q$ simply means $\omega_1Q_1+\cdots+\omega_kQ_k.$

The Betti numbers\footnote{Hereafter all homology and cohomology groups will be with $\mathbb{Z}_2$ coefficients.} of the sets $\Omega^1\supset\Omega^2\supset\cdots\supset\Omega^{n}$ together with $\mu$ are the ingredients we need and we collect them into a table $E=(e_{i,j})$ defined for $i=0, \ldots, k$ and $j=0, \ldots,n-1$ by:
$$e_{i,j}=b_i(W, \Omega^{j+1}).$$
The zero-th and the $k$-th columns look like (there are $\mu$ zeros in the first column and $n-\nu$ in the second one):
$$e_{0,*}=\begin{array}{|c|}
\hline
1\\
\hline
\vdots\\
\hline
1\\
\hline
0\\
\hline
\vdots\\
\hline
0\\
\hline
\end{array}\quad\quad
e_{k,*}=\begin{array}{|c|}
\hline
0\\
\hline
\vdots\\
\hline
0\\
\hline
1\\
\hline
\vdots\\
\hline
1\\
\hline
\end{array}
$$ 

For the first column: if $j\geq\mu$ then the set $\Omega^j$ is empty and $e_{0, j}=b_0(W, \Omega^{j+1})=1;$ on the other hand if $j\leq \mu-1$ then $\Omega^{j+1}\neq \emptyset$ and $e_{0,j}=b_0(W, \Omega^{j+1})=0.$ For the $(k+1)$-th column: if $j\leq \nu-1$ the $\Omega^{j+1}=S^k$ (\emph{every} point $\omega$ in $S^{k-1}$ has $\ii(\omega Q)\geq \nu$); hence $e_{k, j}=b_{k}(W, S^{k-1})=b_{k-1}(S^{k-1})=1.$ On the opposite if $j\geq \mu$ then there is at least a point on the sphere not in $\Omega^{j+1}$ (say a point where the index is minimum); in other words  $\Omega^{j+1}$ is a proper open subset of the sphere $S^{k-1}$ and $e_{k, j}=b_{k}(W, \Omega^j)=b_{k-1}(\Omega^j)=0$ (proper open subsets of $S^{k-1}$ do not have homology in dimension $k-1$). 

Thus the table $E$ can be viewed as a partitioned table (blank spots are zeros):
$$E=\begin{array}{|c|ccc|c|}
\hline
1&&&&0\\
\vdots&&&&\vdots\\
1&&&&0\\
\hline
0&&&&0\\
\vdots&&S &&\vdots\\
0&&&&0\\
\hline
0&&&&1\\
\vdots&&&&\vdots\\
0&&&&1\\
\hline
\end{array}
$$

and the table $S$ is the following:
$$
S=\begin{array}{|c|c|c|c|}
\hline

\hline
 b_{0}(\Omega^{\mu})-1&b_1(\Omega^{\mu})&\cdots&b_{k-2}(\Omega^\mu)\\
\hline
\vdots&\vdots&&\vdots\\

\hline
 b_{0}(\Omega^{\nu+1})-1&b_1(\Omega^{\nu+1})&\cdots&b_{k-2}(\Omega^{\nu+1})\\
\hline

\end{array}
$$
(the $-1$ appear only in the first column). 

Let us denote by $b(S)$ the sum of all the numbers in the table $S$ (and similarly by $b(E)$ the sum of all the numbers in the table $E$).
The following properties hold for $S$.
\begin{prop}\label{prop:bettitwo}
For the generic choice of $q_1, \ldots q_k$ we have $\nu=n-\mu$; moreover if $k=2$
$$b(S)=n-2\mu+\frac{1}{2}\emph{\textrm{Card}}(W\cap \Sigma_{2,n})$$
\end{prop}
\begin{proof}
For the generic choice of $q_1, \ldots, q_k$ the set $\Sigma_{2,n}\cap W$ is a proper algebraic set (it is by definition the set of points in $S^{k-1}\subset W$ where the determinant is zero) and each set $\Omega^j$ is an open set; thus the minimum and the maximum of the index are attained at points where the determinant doesn't vanish; call these points $Q_M$ and $Q_m$ respectively. Then, since $\det(Q_M)\neq 0 \neq \det(Q_m)$, we have:
$$n-\mu =n-\ii(Q_M)\leq n-\ii(-Q_m)=\nu\leq \ii(-Q_M)=n-\mu.$$
Let us move to the second part of the statement. In the case $k=2$, for the generic choice of $q_1, q_2$ we have $\nu=n-\mu$ and the table $S$ is given by one single column $$
S=\begin{array}{|c|}
\hline

\hline
 b_{0}(\Omega^{\mu})-1\\
\hline
\vdots\\

\hline
 b_{0}(\Omega^{n-\mu+1})-1\\
\hline

\end{array}
$$
Each set $\Omega^j$ is a disjoint union of open intervals of $S^{1}$ (it is a proper subset of the circle) and $b_0(\Omega^j)=\frac{1}{2}b_0(\partial \Omega^j).$ In particular:
$$b(S)=n-2\mu+\frac{1}{2}\sum_{j=n-\mu+1}^{\mu }b_0(\partial \Omega^j).$$
Consider now the sum $\sum_{j=n-\mu+1}^{\mu} b_0(\partial \Omega^j)$: for the generic choice of $q_1, q_2$, it equals the number of points $\omega$ on the circle $S^1$ where $\ii(\omega Q)$ changes its value. This happens exactly at the points where the determinant vanishes, i.e. at the points in $W\cap \Sigma_{2, n}$. In other words:
$$\bigcup_{j=n-\mu+1}^{\mu} \partial \Omega^j=W\cap \Sigma_{2,n}.$$

\end{proof}
Given the table $E$, for every $i\geq 0$ we define $b_i(E)$ by the sum of the elements on its $(n-1-i)$-th diagonal\footnote{The strange but standard indexing is due to Alexander-Pontryiagin duality.}, i.e.:
$$b_i(E)=e_{0,n-1-i}+e_{1, n-i}+\cdots + e_{n-i, 1}+e_{n-1-i, 0}.$$
The following theorem is a combination of  \cite[Thm. A]{AgrachevLerario} and \cite[Thm. 8]{Lerariotwo}. 
\begin{thm}\label{bettiq}The entries of $E$ are related to the Betti numbers of the common zero locus $X\subset \RP^{n-1}$ of the generic quadrics $q_1, \ldots, q_k$ by:
 $$b_i(X)\leq b_i(E)\quad \textrm{and}\quad\chi(X)=\sum_{i=0}^{n-1}(-1)^ib_i(E).$$
 Moreover assume for every \emph{nonzero} entry $e_{j, n-1-i-j}$ of the $(n-1-i)$-th diagonal we have:
 \begin{equation}\label{knight}\sum_{t>j+2}e_{t, n-t-i}=0\quad \textrm{and}\quad \sum_{t<j-2}e_{t, n-t-2}=0\end{equation}
 then:
 $$b_i(X)=b_i(E).$$
 In the case $k=2$ we also have:
 \begin{equation}\label{two}|b(X)-b(E)|\leq2.\end{equation}

\end{thm}
Let us explain property \eqref{knight} with a visual example. Let us pick an element $e_{i, j}$ (which belongs to the $(i+j)$-th diagonal and will give information on the $(n-1-i-j)$-th Betti number) and look at the left side of the diagonal below it and the right side of the one above it:

$$\begin{array}{ccccccc}
\cline{1-1}
\multicolumn{1}{|c|}{e_{i-3, j+2}}&&&&&&\\
\cline{1-4}
&\multicolumn{1}{|c|}{e_{i-2, j+1}}&\multicolumn{1}{|c|}{*}&\multicolumn{1}{|c|}{*}&&&\\
\cline{2-5}
&&\multicolumn{1}{|c|}{*}&\multicolumn{1}{|c|}{e_{i,j}}&\multicolumn{1}{|c|}{*}&&\\
\cline{3-6}
&&&\multicolumn{1}{|c|}{*}&\multicolumn{1}{|c|}{*}&\multicolumn{1}{|c|}{e_{i+2, j-1}}&\\
\cline{4-7}
&&&&&&\multicolumn{1}{|c|}{e_{i+3, j-2}}\\
\cline{7-7}
\end{array}
$$

If the sum of all these elements is zero we say that $e_{i,j}$ \emph{survives}. If every nonzero element in a diagonal survives, than the the sum of all of them equals the corresponding Betti number. In this example if every nonzero element in the diagonal containing $e_{i,j}$ survives, then $b_{n-1-i-j}(X)=b_{n-1-i-j}(E).$
\begin{remark}
The table $E$ we have introduced is the table of the ranks of the second term of a spectral sequence converging to the homology of $X$. A spectral sequence is nothing more than a sequence of groups.
 In this context ``converging'' to the homology of $X$ means that these ranks give an upper bound for the homology of $X$ and the fact that they come from a spectral sequence is what implies the two above properties; the reason for the appearance of the ``$2$'' in the above summation is precisely the fact that $E$ corresponds to the \emph{second} term of the spectral sequence (see \cite{BottTu} for more details). 
 \end{remark}
 
 \begin{example}[Empty intersection of two quadrics in $\RP^2$] Consider the quadratic forms:
 $$q_1(x)=x_0^2+x_1^2+x_2^2\quad \textrm{and}\quad q_2(x)=2x_0x_1+2x_2x_1+2x_0x_2+x_1^2 $$
 with corresponding symmetric matrices: 
 $$Q_1=\left(\begin{array}{ccc}1 &0& 0 \\0 & 1 & 0 \\0 & 0 & 1\end{array}\right)\quad \textrm{and}\quad Q_2=\left(\begin{array}{ccc}0 & 1 & 1 \\1 & 1 & 1 \\1 & 1 & 0\end{array}\right) $$
 In this case the index function over the circle $S^1\subset \textrm{span}\{Q_1, Q_2\}$ looks as in Figure \ref{ciao}.
 Consequently $\mu=3$ and the table $E$ consists of all zeros; this is confirmed by the fact that the common zero locus of $q_1$ and $q_2$ is empty ($q_1$ is positive definite).

  \begin{center}
  \begin{figure}
\scalebox{0.8} 
{
\begin{pspicture}(0,-3.1829102)(12.355,3.1829102)
\pscircle[linewidth=0.04,dimen=outer](3.135,-0.35548827){2.34}
\psline[linewidth=0.03cm,linestyle=dashed,dash=0.16cm 0.16cm](1.395,2.5445118)(4.775,-2.9154882)
\psline[linewidth=0.03cm,linestyle=dashed,dash=0.16cm 0.16cm](4.655,2.5645118)(1.615,-3.0154884)
\psline[linewidth=0.03cm,linestyle=dashed,dash=0.16cm 0.16cm](0.015,1.0645118)(6.175,-1.5354882)
\psdots[dotsize=0.12](1.955,1.6245117)
\psdots[dotsize=0.12](1.035,0.6245117)
\psdots[dotsize=0.12](4.175,1.7045118)
\psdots[dotsize=0.12](5.315,-1.1554883)
\psdots[dotsize=0.12](4.395,-2.2954884)
\psdots[dotsize=0.12](1.975,-2.3554883)
\psline[linewidth=0.04cm,arrowsize=0.05291667cm 2.0,arrowlength=1.4,arrowinset=0.4]{->}(3.135,-0.25548828)(3.175,3.0245118)
\psline[linewidth=0.04cm,arrowsize=0.05291667cm 2.0,arrowlength=1.4,arrowinset=0.4]{->}(3.115,-0.23548828)(6.435,-0.21548828)
\usefont{T1}{ptm}{m}{n}
\rput(3.7864552,2.9895117){$Q_1$}
\usefont{T1}{ptm}{m}{n}
\rput(6.986455,-0.11048828){$Q_2$}
\usefont{T1}{ptm}{m}{n}
\rput(2.646455,2.5295117){$3$}
\usefont{T1}{ptm}{m}{n}
\rput(1.0064551,1.5695118){$2$}
\usefont{T1}{ptm}{m}{n}
\rput(0.64645505,-1.5704883){$1$}
\usefont{T1}{ptm}{m}{n}
\rput(3.226455,-3.0104883){$0$}
\usefont{T1}{ptm}{m}{n}
\rput(5.386455,-2.1304884){$1$}
\usefont{T1}{ptm}{m}{n}
\rput(5.486455,1.2095118){$2$}
\psarc[linewidth=0.04](11.805,1.9745117){0.47}{299.57785}{238.44861}
\psarc[linewidth=0.04](11.805,-0.20548828){0.47}{345.25644}{170.83765}
\psarc[linewidth=0.04](11.865,-2.2054882){0.47}{25.709953}{152.35402}
\usefont{T1}{ptm}{m}{n}
\rput(10.356455,2.0295117){$\Omega^1=$}
\usefont{T1}{ptm}{m}{n}
\rput(10.356455,-0.07048828){$\Omega^2=$}
\usefont{T1}{ptm}{m}{n}
\rput(10.356455,-2.0304883){$\Omega^3=$}
\psdots[dotsize=0.12,fillstyle=solid,dotstyle=o](11.575,1.5645118)
\psdots[dotsize=0.12,fillstyle=solid,dotstyle=o](12.035,1.5645118)
\psdots[dotsize=0.12,fillstyle=solid,dotstyle=o](11.335,-0.13548829)
\psdots[dotsize=0.12,fillstyle=solid,dotstyle=o](12.255,-0.31548828)
\psdots[dotsize=0.12,fillstyle=solid,dotstyle=o](11.455,-1.9754883)
\psdots[dotsize=0.12,fillstyle=solid,dotstyle=o](12.275,-1.9754883)
\end{pspicture} 
}
\caption{The index function for the empty intersection of two quadrics.}
\label{ciao}
\end{figure}
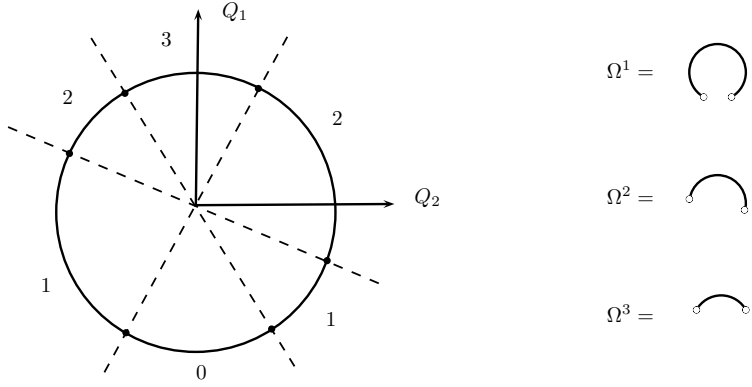
\end{center}

\end{example}
\subsection{The expectation of the maximum of the index}
Fix $k$, and let $W = \Span \{q_1,q_2,..,q_k \}$ be the span of $k$ (Kostlan distributed) random quadrics.
This translates to a span of random matrices $\{Q_1,Q_2,..,Q_k \}$ in GOE.
By homogeneity, we may assume the matrices have been rescaled by $\frac{1}{\sqrt{n}}$.
Let $\mu$ be as above the maximum of the index function $\ii(\omega Q)$ over $\omega \in S^{k-1}$ with $Q = (Q_1,Q_2,..,Q_k)$.

\begin{lemma}\label{lemma:mu}
For every $\alpha,M>0$
$$\PP \left\{ \mu \geq \frac{n}{2} + n^\alpha \right\} = O(n^{-M}),$$ 
for every $M>0$.
\end{lemma}

\begin{proof}[Proof of Lemma]
Fix $\alpha>0$, and let $E_n$ denote the event $$E_n := \left\{ \mu \geq \frac{n}{2} + n^\alpha \right\} ,$$
and 
$$A_n:= \left\{ \|Q_i \| \leq 4\sqrt{n}, i=1,2,..,k \right\}.$$

Using the bound $\|Q_i \| \leq \sqrt{n} \| Q_i \|_{\op}$
along with the fact that $\PP \{ \| Q_i \|_{\op} \geq 4 \}$ is exponentially small \cite{Tao}[Ch. 2], we have:
\begin{equation}\label{eq:events}
\PP \left\{ E_n \right\} = \PP \left\{ E_n \cap A_n \right\} + O(c_1^n),
\end{equation}
for some $0 < c_1<1$.

Let $\N_n$ be a maximal $\e_n$-net of points on the sphere $S^{k-1}$ ($\e_n > 0$ is specified later),
i.e. $\N_n$ is a finite set of points $\omega_i \in S^{k-1}$ that are separated from each other by a distance of at least $\e_n$,
and $\N_n$ is maximal with respect to set inclusion.

\noindent {\bf Claim 1:} Let 
$$E_{n,j}:=\left \{ \ii(\omega_j Q) \geq \frac{n}{2} + n^\alpha/2 \right\}. $$ 
and 
$$ R_{n,j} :=  \{ \omega_j Q \text{ has } \geq n^\alpha/2 \text{ eigenvalues in } (-\e_n 4 \sqrt{nk}, \e_n 4 \sqrt{nk}) \} . $$
Then: 
$$E_n \cap A_n \subset \cup_{\omega_j \in \N_n} \{ E_{n,j} \cup R_{n,j} \}. $$

\noindent {\bf Claim 2:} For each $M>0$, we have: $$ \PP \{ E_{n,j} \} = O(n^{-M}).$$
If $\e_n < n^{-3/2}$ then we also have:
$$ \PP \{ R_{n,j} \} = O(n^{-M}).$$

Choosing, for instance, $\e_n = 1/n^2$ and applying both Claims we have
$$\PP \{ E_n \cup A_n \} < \sum_{j} ( \PP E_{n,j} + \PP R_{n,j} ) = O(n^{-M + 2k}),$$
since the number of terms in the sum is $|\N_n|=O(1/\e_n^k) = O(1/n^{2k})$. 
This proves the result in light of (\ref{eq:events}).

It remains to prove the Claims.

To see Claim 1, assume the events $E_n$ and $A_n$ occur and note that for some $\omega \in S^{k-1}$ we have $\mu = \ii(\omega Q)$ and
this $\omega$ is at most $\e_n$ away from one of the points $\omega_j \in \N_n$.
The change in index $|\ii(\omega Q) - \ii(\omega_j Q)|$ counts the number of changes in sign of the eigenvalues.
If the $i$th eigenvalue $\lambda_i$ changes sign, then $|\lambda_i(\omega Q) - \lambda_i(\omega_j Q )| > |\lambda_i(\omega_j Q )|$.
The Weilandt-Hoffman estimate \cite{Tao}[Ch. 1] states:
$$\sum_{i=1}^n |\lambda_i(\omega Q) - \lambda_i(\omega_j Q )|^2 < \| \omega Q - \omega_j Q \|^2.$$
In particular, for an eigenvalue that changes sign:
$$ |\lambda_i(\omega_j Q )| < |\lambda_i(\omega Q) - \lambda_i(\omega_j Q )| < \| \omega Q - \omega_j Q \| \leq \e_n 4 \sqrt{nk},$$
where the last inequality is implied by the event $A_n$.

This implies that the change in the index $|\ii(\omega Q) - \ii(\omega_j Q )|$
is at most the number of eigenvalues of $\omega_j Q$ with absolute value less than $\e_n 4\sqrt{nk}.$
Thus, either $$|\ii(\omega Q) - \ii(\omega_j Q )| \leq n^\alpha/2,$$
or $\omega_j Q $ has at least $n^\alpha/2$ eigenvalues in $ (-\e_n 4\sqrt{nk} ,\e_n 4\sqrt{nk} ). $
These two possibilities imply the events $E_{n,j}$ and $R_{n,j}$ respectively, so this proves Claim 1.

In order to prove Claim 2, we state a result providing a uniform convergence estimate for Wigner's semi-circle law.
We state just a special case of the result from \cite{EYY2012}:

\begin{thm}[L. Erd\"os, H-T. Yau, J. Yin, 2012]\label{thm:EYY}
Let $X$ be a (rescaled) random matrix in GOE(n), and assume $|b_n| < 5$.
There exist positive constants $A>1$, $C,c,$ and $\phi < 1$,
such that 
$$\PP \left\{ |N_X (-\infty,b_n) - n \cdot m (-\infty,b_n) | \geq (\log n)^L \right\} < C \exp \{ - c (\log n)^{\phi L} \},$$
where $ L = A \log \log n,$ $N_X (-\infty,b_n)$ counts the number of eigenvalues of $X$ in the interval $(-\infty,b_n)$,
and $m$ is the semi-circle measure.
\end{thm}
Let us first apply this to event $E_{n,j}$. We have:
$$ \PP \{ E_{n,j} \} =  \PP \{ | N_X (-\infty,0) - n/2 | \geq n^\alpha/2 \} < \PP \{ | N_X (-\infty,0) - n/2 | \geq (\log n)^L \},$$
for all large enough $n$.
So, $\PP \{ E_{n,j} \} < C \exp \{ - c (\log n)^{\phi L} \} $.
For any $M>0$, eventually (whatever the values of the constants):
$$ c (\log n)^{\phi L - 1} > M. $$
Thus, $$ C \exp \{ - c (\log n)^{\phi L} \} = O( \exp \{ - M (\log n) \}) = O(n^{-M}).$$

Next we apply the same theorem to the event $R_{n,j}$.
Let $b_n = \e_n 4 \sqrt{n k}$, and $X = \omega_j Q$ (which is a matrix in GOE). First, we have:
$$ R_{n,j} = \{ | N_X (-\infty,b_n) - N_X(-\infty,-b_n) | \geq n^\alpha/2 \}.$$
By the triangle inequality,
$$ | N_X (-\infty,b_n) - N_X(-\infty,-b_n) | \leq | N_X (-\infty,b_n) - m(-\infty,b_n) | + |N_X(-\infty,-b_n) - m(-\infty,-b_n)|.$$
Accordingly, $R_{n,j} \subset \hat{R}_{n,j} \cup \tilde{R}_{n,j}$, where
$$\hat{R}_{n,j} := \{ | N_X (-\infty,b_n) - m(-\infty,b_n) | \geq n^\alpha /4 \},$$
$$\tilde{R}_{n,j} := \{  |N_X(-\infty,-b_n) - m(-\infty,-b_n)| \geq n^\alpha /4 \} .$$
Applying Theorem \ref{thm:EYY} to $\PP \{ \hat{R}_{n,j} \}$ and $\PP  \{ \tilde{R}_{n,j} \}$, we have:
$$ \PP \{ R_{n,j} \}  \leq  \PP \{ \hat{R}_{n,j} \} +  \PP  \{ \tilde{R}_{n,j} \} = O(n^{-M}).$$

\end{proof}

As a corollary we derive the following proposition, which computes the expectation of $\mu.$

\begin{prop}\label{propo:mu}For every $k$-dimensional $W\subset \emph{\textrm{GOE}}(n)$ as above and every $0<\alpha<1$ we have:
$$\mathbb{E}\mu=\frac{n}{2}+O(n^\alpha)$$

\end{prop}
\begin{proof}
From one hand we have $\mu\geq n/2$, which gives $\EE \mu\geq n/2.$ On the other hand:
\begin{align*}\EE\mu&=\sum_{j\geq0}\left(\frac{n}{2}+j\right)\PP\left\{\mu=\frac{n}{2}+j\right\}\\
&\leq \left(\frac{n}{2}+n^\alpha\right)\PP\left\{\mu\leq \frac{n}{2}+n^\alpha\right\}+n\PP\left\{\mu>\frac{n}{2}+n^{\alpha}\right\}\\
&\leq \frac{n}{2}+n^\alpha+O(n^{1-M})
\end{align*}
where in the last inequality we have used the assertion of Lemma \ref{lemma:mu}.
\end{proof}

\subsection{Proof of Theorem \ref{thm:bettiquadrics}} Fix $k$ and let $i=i(n)$ be as in the statement. Then as $n$ goes to infinity the table $E$ has a fixed number of columns ($k+1$) and the number of rows is increasing; moreover Betti numbers of $X$ can be studied (deterministically) using Theorem \ref{bettiq}.

Let us start by proving the following property:
\begin{equation}\label{property}b_i(X)=1\quad \textrm{for all $i<n-\mu-k-2$}\end{equation}
(in fact for the generic $X$, because of Poincar\'e duality the same will hold for $i>\mu+2$). 
Let us consider the table $E$  and focus on the diagonal above the one containing $e_{0, n-1-i}$:

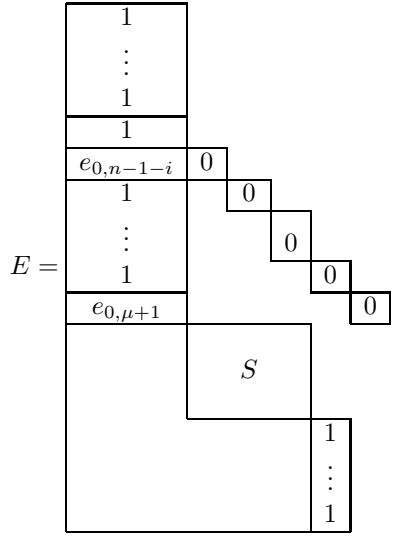
$$E=\begin{array}{|cccccc}
\cline{1-1}
\multicolumn{1}{|c|}{1}&&&&&\\
\multicolumn{1}{|c|}{\vdots}&&&&&\\
\multicolumn{1}{|c|}{1}&&&&&\\

\cline{1-1}
\multicolumn{1}{|c|}{1}&&&&&\\
\cline{1-2}
\multicolumn{1}{|c|}{e_{0, n-1-i}}&\multicolumn{1}{|c|}{0}&&&&\\
\cline{1-3}
\multicolumn{1}{|c|}{1}&&\multicolumn{1}{|c|}{0}&&&\\
\cline{3-4}
\multicolumn{1}{|c|}{\vdots}&&&\multicolumn{1}{|c|}{0}&&\\
\cline{4-5}
\multicolumn{1}{|c|}{1}&&&&\multicolumn{1}{|c|}{0}&\\
\cline{1-1}\cline{5-6}
\multicolumn{1}{|c|}{e_{0, \mu+1}}&&&&&\multicolumn{1}{|c|}{0}\\
\cline{1-4}\cline{6-6}
\multicolumn{1}{|c|}{}&&&\multicolumn{1}{c|}{}&&\\
\multicolumn{1}{|c|}{}&&S &\multicolumn{1}{c|}{}&&\\
\multicolumn{1}{|c|}{}&&&\multicolumn{1}{c|}{}&&\\
\cline{2-5}
&&&&\multicolumn{1}{|c|}{1}&\\
&&&&\multicolumn{1}{|c|}{\vdots}&\\
&&&&\multicolumn{1}{|c|}{1}&\\
\cline{1-5}
\end{array}
$$
Since $n-\mu-i>k+2,$ then this diagonal consists of \emph{all} zeros (except $e_{0, n-i}$): in fact since the number of ones below $e_{0, n-1-i}$ is bigger than the number of columns of $S$, the diagonal $(e_{0, n-i}, e_{1, n-i-1}, \ldots, e_{n-i, 0})$ does not hit $S$, nor the ones in the last columns. This implies:
 $$\sum_{t>2}e_{t, n-t-i}=0\quad \textrm{and}\quad \sum_{t<-2}e_{t, n-t-2}=0$$
 (the second condition  is automatically satisfied because for negative indices $i,j$ the numbers $e_{i,j}$ are zeros). In particular for such an $i$ the second part of Theorem \ref{bettiq} implies:
 $$b_i(X)=b_i(E)=1.$$

 Let now $0< \alpha<1$ and $i=i(n)\leq \lfloor \frac{n}{2}-n^{\alpha}\rfloor$; recall that by Lemma \ref{lemma:mu} we have:
$$\PP\left\{\mu\leq \frac{n}{2}+n^\alpha\right\}\geq 1-O(n^{-M})\quad \textrm{for all $M>0$}.$$
Interpreting $b_i$ as a nonnegative function on $\textrm{Sym}(n, \RR)^k$ endowed with the probability distribution $d\gamma_k$ arising from $\textrm{GOE}(n)^k$, this enables to compute for $\alpha'<\alpha$:
\begin{align*}\EE b_i&=\int_{\textrm{Sym}(n, \RR)^k}b_i d\gamma_k=\int_{\{\mu>\frac{n}{2}+n^{\alpha'}\}}b_id\gamma_k+\int_{\{\mu\leq \frac{n}{2}+n^{\alpha'}\}}b_i d\gamma_k\\
&\geq \int_{\{\mu\leq \frac{n}{2}+n^{\alpha'}\}} d\gamma_k\geq 1+O(n^{-M'}).
\end{align*}
In the previous chain of inequalities we have used the fact that $\mu\leq \frac{n}{2}+n^{\alpha'}$ combined with $i=i(n)\leq \frac{n}{2}-n^\alpha$ gives (for large enough $n$):
$$n-\mu-k-2\geq n-\frac{n}{2}-n^{\alpha'}-k-2=\frac{n}{2}-n^{\alpha'}-k-2\geq \frac{n}{2}-n^{\alpha}\geq i(n)$$
which in turn implies, because of \eqref{property}, that  for each $j\leq i(n)$ on $\{\mu\leq \frac{n}{2}+n^{\alpha'}\}$ one has $b_j\equiv1$.

This proves that for every $0<\alpha<1$ and $M>0$ if $i=i(n)\leq \lfloor \frac{n}{2}-n^{\alpha}\rfloor$ we have $\EE b_i(X)\geq 1+O(n^{-M})$.

For the opposite inequality we need to know the following uniform bound from real algebraic geometry \cite{Lerariocomplexity} on the \emph{sum} of the Betti numbers of an intersection $X$ of $k$ quadrics in $\RP^n$:
\begin{equation}\label{barvinok}
b_i(X)\leq b(X)\leq O(n^{k-1}).
\end{equation}
Reasoning as above we obtain:
\begin{align*}\EE b_i&=\int_{\textrm{Sym}(n, \RR)^k}b_i d\gamma_k=\int_{\{\mu>\frac{n}{2}+n^\alpha\}}b_id\gamma_k+\int_{\{\mu\leq \frac{n}{2}+n^{\alpha}\}}b_i d\gamma_k\\
&\leq O(n)^{k-1}\PP\left\{\mu> \frac{n}{2}+n^\alpha\right\} +\PP\left\{\mu\leq \frac{n}{2}+n^\alpha\right\}\\
& \leq O(n^{-M+k-1})+1 \quad\textrm{for all $M>0$}.
\end{align*}
In particular we have proved that for every $0<\alpha<1$ and $M>0$ if we take $i< \lfloor \frac{n}{2}-n^\alpha\rfloor$ we have:
$$\EE b_i(X)=1+O(n^{-M}).$$
Poincar\'e duality implies that the same holds true for $i>\lfloor \frac{n}{2} +n^{\alpha}\rfloor.$
\subsection{Proof of Theorem \ref{bettirandomtwo}}The proof goes along the lines of the proof of Theorem 1 from \cite{Lerario2012}. In the case $k=2$ we can use the last part of the statement of Theorem \ref{bettiq} and write:
$$b(X)=b(E)+O(1).$$
In this case the first column of $E$ has $n-\mu$ ones and the last has $\nu$ many; but for the generic choice of $q_1, q_2$  Proposition \ref{prop:bettitwo} implies $\nu=n-\mu$; in particular:
\begin{equation}\label{1}b(E)=n-\mu+b(S)+n-\mu=2n-2\mu+b(S)=3n-4\mu+\frac{1}{2}\textrm{Card}(W\cap \Sigma_{2,n})\end{equation}
(here $W=\textrm{span}\{q_1, q_2\}$).
The expectation of $\mu$ is computed in Proposition \ref{propo:mu} and provides:
\begin{equation}\label{2}\EE 4\mu=2n+O(n^{\alpha})\quad \textrm{for all $\alpha>0$}.\end{equation}

It remains to compute $\frac{1}{2}\EE\textrm{Card}(W\cap \Sigma_{2,n})$. We start by noticing that by assumption for every $g\in SO(N)$ the random quadratic forms $q$ and $gq$ have the same distribution (here $N=\dim \textrm{Sym}(n, \RR)$ and the action is not by change of variable, but directly on the space of the coefficients). Thus we have:
$$\EE\textrm{Card}(W\cap \Sigma_{2,n})=\frac{\int_{SO(N)}\EE\textrm{Card}(gW\cap \Sigma_{2,n})dg}{|SO(N)|}=\frac{\mathbb{E}\int_{SO(N)}\textrm{Card}(gW\cap \Sigma_{2,n})dg}{|SO(N)|}=\frac{2|\Sigma_{2,n}|}{|S^{N-2}|}.$$
The first equality is because for every $g\in SO(N)$ we have $\mathbb{E}\textrm{Card}(gW\cap \Sigma_{2,n})=\mathbb{E}\textrm{Card}(W\cap \Sigma_{2,n})$; the second is just linearity of expectation, and the third one is the integral geometry formula \cite{EdelmanKostlan95} (there is no expected value because the integral is constant).

In particular, using Corollary \ref{cor:volume} and the remark immediately after it, we obtain:
\begin{equation}\label{3}\frac{1}{2}\EE\textrm{Card}(W\cap \Sigma_{2,n})=\frac{|\Sigma_{2,n}|}{|S^{N-2}|}=\frac{2}{\sqrt{\pi}}n^{1/2}+O(1).\end{equation}
Combining now \eqref{2} and \eqref{3} into the expectation of \eqref{1} we get the result.

\end{document}